\documentclass[12pt,reqno]{amsart}
\usepackage{amsmath} 

\usepackage{amsthm}
\usepackage{amssymb}
\usepackage{graphics}
\usepackage{latexsym}
\usepackage{comment}
\usepackage{lineno}
\usepackage{color}

\numberwithin{equation}{section}
\newtheorem{thm}{Theorem}[section]
\newtheorem{prop}[thm]{Proposition}
\newtheorem{lemm}[thm]{Lemma}
\newtheorem{cor}[thm]{Corollary}

\theoremstyle{remark}
\newtheorem{rem}{Remark}[section]
\newtheorem{defn}{Definition}

\newcommand{\BBB}{\mathbb}
\newcommand{\R}{{\BBB R}}
\newcommand{\Z}{{\BBB Z}}

\newcommand{\N}{{\BBB N}}
\newcommand{\C}{{\BBB C}}
\newcommand{\s}{\mathcal{S}}

\newcommand{\ee}{\mbox{\boldmath $1$}}

\newcommand{\LR}[1]{{\langle {#1} \rangle }}

\newcommand{\cross}{\times}

\newcommand{\al}{\alpha}
\newcommand{\be}{\beta}
\newcommand{\ga}{\gamma}

\newcommand{\eps}{\varepsilon}

\newcommand{\ta}{\tau}

\newcommand{\p}{\partial}

\newcommand{\supp}{\operatorname{supp}}

\newcommand{\kuuhaku}{\text{}}

\newcommand{\F}{\mathcal{F}}
\newcommand{\h}{\mathcal{H}}

\newcommand{\It}{\mathcal{I}}
\newcommand{\dx}{\partial_x}

\newcommand{\II}{I\hspace{-0.03cm}I}

\DeclareMathOperator*{\dist}{dist}

\usepackage{hyperref}

\title[Well-posedness for a system of qDNLS]{Well-posedness 
for a system of quadratic derivative nonlinear Schr\"odinger 
equations in almost critical spaces
}

\author[H. Hirayama]{Hiroyuki Hirayama}
\address[H. Hirayama]{Institute for Tenure Track Promotion, University of Miyazaki, 1-1, Gakuenkibanadai-nishi, Miyazaki, 889-2192 Japan}
\email[H. Hirayama]{h.hirayama@cc.miyazaki-u.ac.jp}

\author[S. Kinoshita]{Shinya Kinoshita}
\address[S. Kinoshita]{Universit\"at Bielefeld, 
Fakult\"at f\"ur Mathematik, Postfach 10 01 31 33501, Bielefeld, Germany
}
\email[S. Kinoshita]{kinoshita@math.uni-bielefeld.de}

\author[M. Okamoto]{Mamoru Okamoto}
\address[M. Okamoto]{Department of Mathematics,
Graduate School of Science, Osaka University,
Toyonaka, Osaka 560-0043, Japan}
\email[M. Okamoto]{okamoto@math.sci.osaka-u.ac.jp}

\keywords{Schr\"odinger equation; Well-posedness; Cauchy problem; Bilinear estimate}

\begin{document}
%
%
\begin{abstract}
In this paper,
we consider the Cauchy problem 
of the system of quadratic derivative nonlinear 
Schr\"odinger equations introduced by Colin and Colin (2004). 
We determine an almost optimal Sobolev regularity
where the smooth flow map of the Cauchy problem exists,
expect for the scaling critical case.
This result covers a gap left open in papers of the first and second authors (2014, 2019). 
\end{abstract}
\maketitle
\setcounter{page}{001}


\section{Introduction\label{intro}}
We consider the Cauchy problem of the system of nonlinear Schr\"odinger equations:
\begin{equation}\label{NLS_sys}
\begin{cases}
\displaystyle (i\p_{t}+\alpha \Delta )u=-(\nabla \cdot w )v,
\ \ t>0,\ x\in \R^d,\\
\displaystyle (i\p_{t}+\beta \Delta )v=-(\nabla \cdot \overline{w})u,
\ \ t>0,\ x\in \R^d,\\
\displaystyle (i\p_{t}+\gamma \Delta )w =\nabla (u\cdot \overline{v}),
\ \ t>0,\ x\in \R^d,\\
(u, v, w)|_{t=0}=(u_{0}, v_{0}, w_{0})\in \h^s(\R^d),
\end{cases}
\end{equation}
where $\al$, $\be$, $\ga\in \R\backslash \{0\}$ 
, the unknown functions $u$, $v$, $w$ are $d$-dimensional complex vector valued.
Moreover, $H^s(\R^d)$ denotes the $L^2$-based Sobolev space,
and we set
\[
\h^s(\R^d) :=(H^s(\R^d))^d\times (H^s(\R^d))^d\times (H^s(\R^d))^d.
\]
Our aim in this paper is to determine regularities where the smooth flow map of \eqref{NLS_sys} exists.

The system (\ref{NLS_sys}) was introduced by Colin and Colin in \cite{CC04} 
as a model of laser-plasma interaction. (See also \cite{CC06}, \cite{CCO09_1}.)
The local existence of the solution of (\ref{NLS_sys}) in $\h^s(\R^d)$ for $s>\frac{d}{2}+3$ is obtained in \cite{CC04}. 
The system (\ref{NLS_sys}) is invariant under the following scaling transformation:
\begin{equation}
\notag
A_{\lambda}(t,x)=\lambda^{-1}A(\lambda^{-2}t,\lambda^{-1}x)\ \ (A=(u,v,w) ), 
\end{equation}
and the scaling critical regularity is $s_{c}=\frac{d}{2}-1$. 
We set
\begin{equation}\label{coeff}
\mu :=\alpha\beta\gamma \left(\frac{1}{\alpha}-\frac{1}{\beta}-\frac{1}{\gamma}\right),\ \ 
\kappa :=(\alpha -\beta)(\alpha -\gamma)(\beta +\gamma). 
\end{equation}
We note that $\kappa =0$ does not occur when $\mu \ge 0$ 
for $\al$, $\be$, $\ga\in \R\backslash \{0\}$. 

First, we mention some known results for related problems. 
Since the system (\ref{NLS_sys}) has quadratic nonlinear terms which contain a derivative,
a derivative loss arising from the nonlinearity makes the problem difficult. 
In fact, Mizohata (\cite{Mi85}) considered the Schr\"odinger equation
\[
\begin{cases}
i\partial_{t}u-\Delta u=(b_{1}(x)\cdot \nabla ) u,\ t\in \R ,\ x\in \R^{d},\\
u(0,x)=u_{0}(x),\ x\in \R^{d}
\end{cases}
\]
and proved that the uniform bound
\[
\sup_{x\in \R^{n},\omega \in S^{n-1},R>0}\left| {\rm Re}\int_{0}^{R}b_{1}(x+r\omega )\cdot \omega dr\right| <\infty
\]
is a necessary condition for the $L^{2} (\R^d)$ well-posedness. 
Furthermore, Christ (\cite{Ch}) proved that the flow map of 
the nonlinear Schr\"odinger equation
\begin{equation}
\notag
\begin{cases}
i\partial_{t}u-\partial_{x}^{2}u=u\partial_{x}u,\ t\in \R,\ x\in \R,\\
u(0,x)=u_{0}(x),\ x\in \R
\end{cases}
\end{equation}
is not continuous on $H^{s} (\R)$ for any $s\in \R$. 
From these results, 
it is difficult to obtain the well-posedness 
for quadratic derivative nonlinear Schr\"odinger equation in general. 
%
See \cite{Por18} and references therein for derivative nonlinear Schr\"odinger equations with cubic or higher order nonlinearities.

Next, we introduce the previous results for (\ref{NLS_sys}). 
In \cite{Hi} and \cite{HK}, the first and second 
authors proved the well-posedness of (\ref{NLS_sys}) 
in $\h^s(\R^d)$ under the condition $\kappa \ne 0$, 
where $s$ is given in Table~\ref{WP_NLS_sys} below. 
\begin{table}[h]
\begin{center}
\begin{tabular}{|l|l|l|l|l|}
\hline
\multicolumn{2}{|c|}{} & \multicolumn{1}{|c|}{$d=1$} & 
\multicolumn{1}{|c|}{$d=2$ or $3$} & 
\multicolumn{1}{|c|}{$d\geq 4$}\\
\hline
\multicolumn{2}{|c|}{$\mu > 0$} & \multicolumn{1}{|c|}{$s\geq 0$} &\multicolumn{1}{|c|}{$s\geq s_{c}$} & \\              
\cline{1-4}
\multicolumn{2}{|c|}{$\mu =0$} & \multicolumn{1}{|c|}{$s\geq 1$} & 
\multicolumn{1}{|c}{$s\geq 1$}
& \multicolumn{1}{|c|}{$s\geq s_{c}$}
\\
\cline{1-4}
\multicolumn{2}{|c|}{$\mu < 0$} & \multicolumn{1}{|c|}{$s\geq \frac{1}{2}$} & 
\multicolumn{1}{|c|}{$s\geq \frac{1}{2}$, $s>s_c$}&\multicolumn{1}{|c|}{}\\
\hline
\end{tabular}
\caption{Regularities to be well-posed in \cite{Hi} and \cite{HK}}
\end{center}
\end{table}\label{WP_NLS_sys}

In \cite{Hi}, the first author also considered 
the case $\kappa =0$ 
and proved the 
well-posedness of \rm (\ref{NLS_sys}) 
in $\h^s(\R^d)$ for $s\ge \frac{1}{2}$ if $d=1$, 
$s>1$ if $d=2$ or $3$, and $s>s_c$ if $d\ge 4$ 
under the condition $\alpha=\beta$ and $(\beta +\gamma)(\gamma-\alpha)=0$. 
On the other hand, 
the first author proved that
the flow map is not $C^2$ for $s<1$ if $\mu= 0$, 
for $s<\frac{1}{2}$ if $\mu< 0$ and 
$(\beta+\gamma)(\gamma -\alpha)\ne 0$, 
and for any $s\in \R$ if $(\beta+\gamma)(\gamma -\alpha)=0$.
Therefore, the well-posedness  
obtained in \cite{Hi} and \cite{HK} 
are optimal except for the following cases 
{\rm (A1)}--{\rm (A4)} 
as far as we use the iteration argument:
\begin{enumerate}
\item[{\rm (A1)}] $d=1$, $\mu >0$, and $s<0$,
\item[{\rm (A2)}] $d=2$ or $3$, $\alpha =\beta$, 
$(\beta+\gamma)(\gamma-\alpha)\ne 0$, and $\frac{1}{2}\le s \le 1$,
\item[{\rm (A3)}] $d=3$, $\kappa \ne 0$, $\mu <0$ and $s=\frac{1}{2}$ (which is scaling critical),
\item[{\rm (A4)}] $d\ge 4$, $\alpha =\beta$,
$(\beta+\gamma)(\gamma-\alpha)=0$, and $s=s_c$.
\end{enumerate}
The radial settings are also considered in \cite{HKO}. 

We point out that the results in \cite{Hi} and \cite{HK} do not 
contain the asymptotic behavior of the solution for 
$d\le 3$ under the condition $\mu =0$ 
(and also $\mu <0$). 
In \cite{IKS13}, Ikeda, Katayama, and Sunagawa considered 
the system of quadratic nonlinear Schr\"odinger equations
\begin{equation}\label{qDNLS}
\left(i\partial_t+\frac{1}{2m_j}\Delta\right)u_j=F_j(u,\partial_xu),\ \ t>0,\ x\in \R^d,\ j=1,2,3, 
\end{equation}
under the mass resonance condition 
$m_1+m_2=m_3$ (which corresponds to the condition $\mu =0$ for (\ref{NLS_sys})), 
where $u=(u_1,u_2,u_3)$ is $\C^3$-valued, 
$m_1$, $m_2$, $m_3\in \R\backslash \{0\}$, and $F_j$ is defined by
\begin{equation}\label{qDNLS_nonlin}
\begin{cases}
F_{1}(u,\partial_xu)=\sum_{|\alpha |, |\beta|\le 1}
C_{1,\alpha,\beta}(\overline{\partial^{\alpha}u_2})(\partial^{\beta}u_3),\\
F_{2}(u,\partial_xu)=\sum_{|\alpha |, |\beta|\le 1}
C_{1,\alpha,\beta}(\partial^{\beta}u_3)(\overline{\partial^{\alpha}u_1}),\\
F_{3}(u,\partial_xu)=\sum_{|\alpha |, |\beta|\le 1}
C_{1,\alpha,\beta}(\partial^{\alpha}u_1)(\partial^{\beta}u_2)
\end{cases}
\end{equation}
with some constants $C_{1,\alpha,\beta}$, $C_{2,\alpha,\beta}$, $C_{3,\alpha,\beta}\in \C$. 
They obtained the small data global existence and the scattering 
of the solution to (\ref{qDNLS})
in the weighted Sobolev space for $d=2$ 
if
the nonlinear terms $F_j$ in \eqref{qDNLS_nonlin} satisfy
the null condition. 
They also proved the same result for $d\ge 3$ without the null condition. 
In \cite{IKO16}, Ikeda, Kishimoto, and third author proved
the small data global well-posedness and the scattering of the solution 
to (\ref{qDNLS}) in $\h^s (\R^d)$ for $d\ge 3$ and $s\ge s_c$ 
under the null condition. 
They also proved the local well-posedness in $\h^s (\R^d)$
for $d=1$ and $s\ge 0$, $d=2$ and $s>s_c$, and $d=3$ and $s\ge s_c$ 
under the same conditions. 
(The results in \cite{Hi} for $d\le 3$ and $\mu =0$ 
say that if the nonlinear terms do not have null condition, 
then $s=1$ is optimal regularity to obtain the well-posedness 
by using the iteration argument.)

In \cite{SaSu}, 
Sakoda and Sunagawa considered 
(\ref{qDNLS}) for $d=2$ and $j=1,\dots, N$ with
\begin{equation}\label{qDNLS_nonlin_2}
F_j(u,\partial_xu)
=\sum_{|\alpha|, |\beta|\le 1}\sum_{1\le k,l\le 2N}
C^{\alpha, \beta}_{j,k,l}(\partial_x^{\alpha}u_k^{\#})(\partial_x^{\beta}u_l^{\#}),
\end{equation}
where $u_k^{\#}=u_k$ if $k=1,\dots ,N$, and 
$u_k^{\#}=\overline{u_k}$ if $k=N+1,\dots ,2N$. 
They obtained the small data global existence and the time decay estimate 
for the solution
under some conditions for $m_1,\dots m_N$ 
and the nonlinear terms (\ref{qDNLS_nonlin_2}).
Note that their argument covered (\ref{NLS_sys}) with $\mu =0$. 
In \cite{KS20}, Katayama and Sakoda 
considered (\ref{qDNLS}) for $d=1$ and $2$ 
with more general nonlinearity 
than (\ref{qDNLS_nonlin_2}), and obtained 
asymptotic behavior of solution for small initial data. 
In particular, they gave the examples of non-scattering 
solutions to (\ref{NLS_sys}) for small initial data 
under the condition $\mu =0$. 
Moreover, it is known that the existence of the blow up solutions 
for the system of nonlinear Schr\"odinger equations. 
Ozawa and Sunagawa (\cite{OS13}) gave the examples of the derivative nonlinearity which causes the small data blow up for a system of Schr\"odinger equations. 
See \cite{HLN11}, \cite{HLO11}, \cite{HOT13} and references therein for a system of nonlinear Schr\"odinger equations without derivative nonlinearity. 

Now, we give our main results. 
The first result is that
the flow map fails to be $C^3$ for the case {\rm (A1)}.
\begin{thm}\label{NotC3}
Let $d=1$, $\mu >0$, and $s<0$. 
Then, the flow map of {\rm (\ref{NLS_sys})} 
is not $C^3$ in $\h^s(\R^d)$. 
\end{thm}

It is known that the flow map is smooth if we prove the well-posedness by using the contraction mapping theorem (or the iteration argument).
While there is a gap between the failure of the smoothness of the flow map and ill-posedness,
Theorem \ref{NotC3} says that the contraction mapping theorem does not work to prove the well-posedness of \eqref{NLS_sys} for $s<0$.

While the nonlinear term in \eqref{NLS_sys} is quadratic,
to show the existence of an irregular flow map,
we need to consider the third iteration term as in the KdV equation (see Section 6 in \cite{Bo97}).

Next, we consider the case {\rm (A2)}. 

\begin{thm}\label{LWP:A2}
Let $d=2$ or $3$.
Assume that $\alpha, \beta, \gamma \in \R \setminus \{ 0 \}$ satisfy
\[
\alpha =\beta \ne 0, \quad
(\beta+\gamma)(\gamma-\alpha)\ne 0.
\]
Then, \eqref{NLS_sys} 
is locally well-posed in $\h^s(\R^d)$ 
for $s\ge \frac{1}{2}$ and $s>s_c$. 
\end{thm}

We mention the difference of Theorem \ref{LWP:A2} and the previous results in \cite{HK}.
In \cite{HK},
they used the fact that the sizes of the frequencies of $u$, $v$, and $w$ are almost the same when the oscillation is small.
On the other hand,
when $\alpha=\beta$,
the frequency of $w$ may be smaller than that of $u$ and $v$ even if the oscillation is small.
See Lemma \ref{modul_est_2} below.
However, since the derivative only hits $w$ in \eqref{NLS_sys},
we can treat this case.

For the proof of Theorem \ref{LWP:A2},
we use the Fourier restriction norm method introduced by Bourgain in \cite{Bo93}.
Namely, we rely on the contraction mapping theorem in the Fourier restriction norm space.
A bilinear estimate in the Fourier restriction norm space plays a key role in the proof.
See Proposition \ref{key_be} below.
Moreover, the flow map obtained in Theorem \ref{LWP:A2} is smooth.
From Theorems \ref{NotC3} and \ref{LWP:A2} with the previous results in \cite{Hi} and \cite{HK},
we obtain a classification of the regularity where the smooth flow map exists except for the scaling critical cases {\rm (A3)} and {\rm (A4)}.

\noindent {\bf Notation.} 
We denote the spatial Fourier transform by\ \ $\widehat{\cdot}$\ \ or $\F_{x}$, 
the Fourier transform in time by $\F_{t}$ and the Fourier transform in all variables by\ \ $\widetilde{\cdot}$\ \ or $\F_{tx}$. 
For $\sigma \in \R$, the free evolution $e^{it\sigma \Delta}$ on $L^{2}$ is given as a Fourier multiplier
\[
\F_{x}[e^{it\sigma \Delta}f](\xi )=e^{-it\sigma |\xi |^{2}}\widehat{f}(\xi ). 
\]
We will use $A\lesssim B$ to denote an estimate of the form $A \le CB$ for some constant $C$ and write $A \sim B$ to mean $A \lesssim B$ and $B \lesssim A$. 
We will use the convention that capital letters denote dyadic numbers, e.g. $N=2^{n}$ for $n\in \N_0:=\N\cup\{0\}$ and for a dyadic summation we write
$\sum_{N}a_{N}:=\sum_{n\in \N_0}a_{2^{n}}$ and $\sum_{N\geq M}a_{N}:=\sum_{n\in \N_0, 2^{n}\geq M}a_{2^{n}}$ for brevity. 
Let $\chi \in C^{\infty}_{0}((-2,2))$ be an even, non-negative function such that $\chi (t)=1$ for $|t|\leq 1$. 
We define $\psi (t):=\chi (t)-\chi (2t)$, 
$\psi_1(t):=\chi (t)$, and $\psi_{N}(t):=\psi (N^{-1}t)$ for $N\ge 2$. 
Then, $\sum_{N}\psi_{N}(t)=1$.  
We define frequency and modulation projections by
\[
\widehat{P_{N}u}(\xi ):=\psi_{N}(\xi )\widehat{u}(\xi ),\ 
\widetilde{Q_{L}^{\sigma}u}(\tau ,\xi ):=\psi_{L}(\tau +\sigma |\xi|^{2})\widetilde{u}(\tau ,\xi ).
\]
Furthermore, we define $Q_{\geq M}^{\sigma}:=\sum_{L\geq M}Q_{L}^{\sigma}$ and $Q_{<M}:=Id -Q_{\geq M}$. 

The rest of this paper is planned as follows.
In Section 2, we restate Theorem \ref{LWP:A2} and collect some results on
the linear and bilinear Strichartz estimates 
and the property of low modulation. 
In Section 3, we prove Proposition~\ref{key_be} 
for $d=2$. 
In Section 4, we show Proposition~\ref{key_be} 
for $d=3$. 
In Section 5, we prove Theorem~\ref{NotC3}. 
\section{Preliminary}
%
%

In this section, we state some preliminary results used in the proof of Theorem \ref{LWP:A2}.
In Subsection \ref{2.1},
by using the condition for coefficients,
we restate Theorem \ref{LWP:A2}. 
We then define the Fourier restriction norm space.
In Subsection \ref{2.2},
we collect some useful lemmas.

\subsection{Fourier restriction norm spaces}
\label{2.1}

First, we rewrite Theorem \ref{LWP:A2} by using a change of variables.
Set $\sigma =\alpha^{-1}\gamma$ and
\[
U(t,x)=\alpha^{-1}u(\alpha^{-1}t,x),\ 
V(t,x)=\alpha^{-1}v(\alpha^{-1}t,x),\ 
W(t,x)=\alpha^{-1}w(\alpha^{-1}t,x).  
\]
Then, (\ref{NLS_sys}) can be rewritten 
\begin{equation}\label{NLS_sys_c}
\begin{cases}
\displaystyle (i\p_{t}+\Delta )U=-(\nabla \cdot W)V,
\ \ t>0,\ x\in \R^d,\\
\displaystyle (i\p_{t}+\Delta )V=-(\nabla \cdot \overline{W})U,
\ \ t>0,\ x\in \R^d,\\
\displaystyle (i\p_{t}+\sigma \Delta )W =\nabla (U\cdot \overline{V}),
\ \ t>0,\ x\in \R^d,\\
(U, V, W)|_{t=0}=(U_{0}, V_{0}, W_{0})\in \h^s(\R^d),
\end{cases}
\end{equation}
and the condition $(\beta+\gamma)(\gamma -\alpha)\ne 0$ is 
equivalent to $\sigma \ne \pm 1$.
Hence, Theorem \ref{LWP:A2} is equivalent to the following.
\begin{thm}\label{WPab}
Let $d=2$ or $3$ and $\sigma \in \R\backslash \{0,\pm1\}$. 
Then, {\rm (\ref{NLS_sys_c})} 
is locally well-posed in $\h^s(\R^d)$ 
for $s\ge \frac{1}{2}$ and $s>s_c$. 
\end{thm}

Now, we define the Fourier restriction norm, 
which was introduced by Bourgain in \cite{Bo93}. 
\begin{defn}
Let $s\in \R$, $b\in \R$, $\sigma \in \R\backslash \{0\}$. 
We define $X^{s,b}_{\sigma}:=\{u\in \s'(\R_t\times \R_x^d)|\ \|u\|_{X^{s,b}_{\sigma}}<\infty\}$,  
where
\[
\begin{split}
\|u\|_{X^{s,b}_{\sigma}}&:=
\|\langle \xi \rangle^s\langle \tau +\sigma |\xi|^2\rangle^b\widetilde{u}(\tau,\xi)\|_{L^2_{\tau\xi}}\\
&\sim \left(\sum_{N\ge 1}
\sum_{L\ge 1}N^{2s}L^{2b}\|Q_{L}^{\sigma}P_{N}u\|_{L^2}^2\right)^{\frac{1}{2}}, 
\end{split}
\]
where $\langle \xi \rangle := (1+|\xi|^2)^{\frac{1}{2}}$, $P_N$ and $Q_L^{\sigma}$ are 
defined in Notation at the end of Section \ref{intro}. 
\end{defn}
The key estimates to obtain 
Theorem~\ref{WPab} are the following. 
\begin{prop}\label{key_be}
Let $d=2,3$, $\sigma \in \R\backslash \{0,\pm 1\}$, $s\ge \frac{1}{2}$, 
$s>s_c$, 
and $j\in \{1,\dots, d\}$.  
Then there exist $b'\in (0,\frac{1}{2})$ and $C>0$ such that
\begin{align}
\|(\p_jW)V\|_{X^{s,-b'}_{1}}
&\le C\|W\|_{X^{s,b'}_{\sigma}}\|V\|_{X^{s,b'}_{1}},
\notag\\
\|(\p_j\overline{W})U\|_{X^{s,-b'}_{1}}
&\le C\|W\|_{X^{s,b'}_{\sigma}}\|U\|_{X^{s,b'}_{1}},
\notag\\
\|\p_j(U\overline{V})\|_{X^{s,-b'}_{\sigma}}
&\le C\|U\|_{X^{s,b'}_{1}}\|V\|_{X^{s,b'}_{1}}
\notag
\end{align}
hold for any $U,V\in X^{s,b'}_{1}$ 
and $W\in X^{s,b'}_{\sigma}$, 
where $\p_j=\frac{\partial}{\partial x_j}$. 
\end{prop}
\begin{rem} \label{duality}
Note that $\| \overline{V} \| _{X^{s,b'}_1} = \| V \| _{X^{s,b'}_{-1}}$.
By the duality argument, 
to obtain Proposition~\ref{key_be}, 
it suffices to show that
\begin{equation}\label{key_est}
\left|\int_{\R}\int_{\R^d}U(t,x)V(t,x)\partial_jW(t,x)dxdt\right|
\lesssim \|U\|_{X^{s_1,b'}_1}\|V\|_{X^{s_2,b'}_{-1}}\|W\|_{X^{s_3,b'}_{-\sigma}}
\end{equation}
for $(s_1,s_2,s_3)\in \{(s,s,-s),(s,-s,s),(-s,s,s)\}$ 
with $s\ge \frac{1}{2}$ and $s>s_c$.
Moreover, Plancherel's theorem yields that the left hand side of \eqref{key_est} is written as follows:
\begin{align*}
&\left|\int_{\R}\int_{\R^d}U(t,x)V(t,x)\partial_jW(t,x)dxdt\right| \\
&= 
\left|\int_{\R}\int_{\R^d} \int_{\R}\int_{\R^d} \widetilde{U}(\tau_1, \xi_1)\widetilde{V}(\tau_2,\xi_2) (\xi_1^{(j)}+\xi_2^{(j)}) \widetilde{W}(-\tau_1-\tau_2,-\xi_1-\xi_2)d\tau_1 d\xi_1 d\tau_2 d\xi_2 \right|,
\end{align*}
where $\xi_1^{(j)}$ and $\xi_2^{(j)}$ 
are the $j$-th components of $\xi_1$ and $\xi_2$,
respectively.
Hence, \eqref{key_est} is equivalent to the following:
\begin{align*}
&\left|\int_{\R}\int_{\R^d} \int_{\R}\int_{\R^d}  (\xi_1^{(j)}+\xi_2^{(j)}) f_1(\tau_1, \xi_1) f_2(\tau_2,\xi_2)f_3 (\tau_1+\tau_2,+\xi_1+\xi_2)d\tau_1 d\xi_1 d\tau_2 d\xi_2 \right| \\
&\lesssim
\| \mathcal{F}^{-1}_{\tau\xi} [f_1]\|_{X^{s_1,b'}_1}\|\mathcal{F}^{-1}_{\tau\xi} [f_2]\|_{X^{s_2,b'}_{-1}}\|\mathcal{F}^{-1}_{\tau\xi} [f_3]\|_{X^{s_3,b'}_{\sigma}}.
\end{align*}
\end{rem}
We note that 
\[
\|u\|_{X^{s,b'}_{\sigma}(T)}\lesssim T^{b-b'}\|u\|_{X^{s,b}_{\sigma}(T)}. 
\]
holds for any $s\in \R$, $\sigma\in \R\backslash \{0\}$, 
$\frac{1}{2}<b\le 1$, 
$0\le b'\le 1-b$, and $0<T\le 1$, 
where $X^{s,b}_{\sigma}(T)$ denotes the time restricted space 
of $X^{s,b}_{\sigma}$. 
Hence,
by using the fixed point argument 
with Proposition~\ref{key_be}, 
we can obtain Theorem~\ref{WPab}.
Since this argument is standard by now,
we omit the details in this paper.
See \cite{Hi} and \cite{HK} for example.

\subsection{Linear and bilinear estimates}
\label{2.2}

In this subsection,
we collect some propositions used in the proof of Proposition \ref{key_be}.
First, we state the Strichartz and bilinear Strichartz estimates.
We say that $(p,q)$ is an admissible pair if $p$ and $q$ satisfy $2 \le p,q \le \infty$, 
$\frac{2}{p}+\frac{d}{q}=\frac{d}{2}$, and $(p,q,d) \neq (2,\infty,2)$.
\begin{prop}[Strichartz estimate (cf. \cite{GV85}, \cite{KT98})]\label{Stri_est}
Let $\sigma \in \R\backslash \{0\}$ and $(p,q)$ be admissible.
Then, we have
\[
\|e^{it\sigma \Delta}\varphi \|_{L_{t}^{p}L_{x}^{q}(\R\times \R^d)}\lesssim \|\varphi \|_{L^{2}_{x}(\R^d)}
\]
for any $\varphi \in L^{2}(\R^{d})$. 
\end{prop}
The Strichartz estimate implies the following. 
See the proof of Lemma\ 2.3 in \cite{GTV97}.

\begin{cor}
\label{Bo_Stri}
\notag
Let $L\in 2^{\N_0}$, $\sigma \in \R\backslash \{0\}$, 
and $(p,q)$ be admissible.
Then, we have
\begin{equation}
\|Q_{L}^{\sigma}u\|_{L_{t}^{p}L_{x}^{q}(\R\times \R^d)}\lesssim L^{\frac{1}{2}}\|Q_{L}^{\sigma}u\|_{L^{2}_{tx}(\R\times \R^d)}
\end{equation}
for any $u \in L^{2}(\R\times \R^{d})$. 
\end{cor}

We have the following bilinear Strichartz estimates.

\begin{prop}\label{L2be}
Let $d \ge 2$, $\sigma_{1}$, $\sigma_{2}\in \R \backslash \{0\}$. 
For any dyadic numbers $N_1$, $N_2$, $N_3\in 2^{\N_0}$ 
and $L_1$, $L_2\in 2^{\N_0}$, we have
\begin{equation}
\notag
\begin{split}
&\|P_{N_3}(Q_{L_1}^{\sigma_1}P_{N_1}u_{1}\cdot Q_{L_2}^{\sigma_2}P_{N_2}u_{2})\|_{L^{2}_{tx}(\R\times \R^d)}\\
&\lesssim (N_{\min}^{12})^{\frac{d}{2}-1}\left(\frac{N_{\min}^{12}}{N_{\max}^{12}}\right)^{\frac{1}{2}}L_1^{\frac{1}{2}}L_2^{\frac{1}{2}}
\|Q_{L_1}^{\sigma_1}P_{N_1}u_{1}\|_{L^2_{tx}(\R\times \R^d)}\|Q_{L_2}^{\sigma_2}P_{N_2}u_{2}\|_{L^2_{tx}(\R\times \R^d)},
\end{split}
\end{equation}
where
$N_{\max}^{12}=N_1\vee N_2$ and 
$N_{\min}^{12}=N_1\wedge N_2$.
Furthermore, if $\sigma_1+\sigma_2\ne 0$, then we have
\begin{equation}
\notag
\begin{split}
&\|P_{N_3}(Q_{L_1}^{\sigma_1}P_{N_1}u_{1}\cdot Q_{L_2}^{\sigma_2}P_{N_2}u_{2})\|_{L^{2}_{tx}(\R\times \R^d)}\\
&\lesssim N_{\min}^{\frac{d}{2}-1}\left(\frac{N_{\min}}{N_{\max}}\right)^{\frac{1}{2}}L_1^{\frac{1}{2}}L_2^{\frac{1}{2}}
\|Q_{L_1}^{\sigma_1}P_{N_1}u_{1}\|_{L^2_{tx}(\R\times \R^d)}\|Q_{L_2}^{\sigma_2}P_{N_2}u_{2}\|_{L^2_{tx}(\R\times \R^d)},
\end{split}
\end{equation}
where $N_{\max}=\displaystyle \max_{1\le j\le 3}N_j$,\ 
$N_{\min}=\displaystyle \min_{1\le j\le 3}N_j$.
\end{prop}
Proposition~\ref{L2be} can be obtained by the same way as 
in the proof of Lemma\ 1 in \cite{CDKS01}.
See also Lemma 3.1 in \cite{Hi}.

An interpolation argument yields the following.
Since the proof is the same as that of Corollary~2.5 in \cite{HK},
we omit the details here.
\begin{cor}\label{L2be_2}
Let $d \ge 2$, $b'\in (\frac{1}{4},\frac{1}{2})$, 
and 
$\sigma_{1}$, $\sigma_{2}\in \R \backslash \{0\}$. 
We set $\delta =\frac{1}{2}-b'$. 
For any dyadic numbers $N_1$, $N_2$, $N_3\in 2^{\N_0}$ 
and $L_1$, $L_2\in 2^{\N_0}$, we have
\begin{equation}\label{L2be_est_2}
\begin{split}
&\|P_{N_3}(Q_{L_1}^{\sigma_1}P_{N_1}u_{1}\cdot Q_{L_2}^{\sigma_2}P_{N_2}u_{2})\|_{L^{2}_{tx}(\R\times \R^d)}\\
&\lesssim 
(N_{\min}^{12})^{\frac{d}{2}-1+4 \delta}
\left(\frac{N_{\min}^{12}}{N_{\max}^{12}}\right)^{\frac{1}{2}- 2\delta}L_1^{b'}L_2^{b'}
\|Q_{L_1}^{\sigma_1}P_{N_1}u_{1}\|_{L^2_{tx}(\R\times \R^d)}\|Q_{L_2}^{\sigma_2}P_{N_2}u_{2}\|_{L^2_{tx}(\R\times \R^d)}.
\end{split}
\end{equation}
Furthermore, if $\sigma_1+\sigma_2\ne 0$, then we have
\begin{equation}\label{L2be_est_20}
\begin{split}
&\|P_{N_3}(Q_{L_1}^{\sigma_1}P_{N_1}u_{1}\cdot Q_{L_2}^{\sigma_2}P_{N_2}u_{2})\|_{L^{2}_{tx}(\R\times \R^d)}\\
&\lesssim 
N_{\min}^{\frac{d}{2}-1+4 \delta}
\left(\frac{N_{\min}}{N_{\max}}\right)^{\frac{1}{2}- 2\delta}L_1^{b'}L_2^{b'}
\|Q_{L_1}^{\sigma_1}P_{N_1}u_{1}\|_{L^2_{tx}(\R\times \R^d)}\|Q_{L_2}^{\sigma_2}P_{N_2}u_{2}\|_{L^2_{tx}(\R\times \R^d)}.
\end{split}
\end{equation}
\end{cor}
%
%
%

Next, we consider 
the low modulation case 
$L_{\textnormal{max}}:=\displaystyle \max_{1\le j\le 3}L_j \ll N_{\max}^2$. 
\begin{lemm}\label{modul_est_2}
Let $\sigma \in \R \setminus \{0,\pm 1\}$. 
We assume that $(\tau_{1},\xi_{1})$, $(\tau_{2}, \xi_{2})$, $(\tau_{3}, \xi_{3})\in \R\times \R^{d}$ satisfy $\tau_{1}+\tau_{2}+\tau_{3}=0$, $\xi_{1}+\xi_{2}+\xi_{3}=0$. 
If it holds that
\[
\max\{|\tau_{1}+|\xi_{1}|^{2}|,|\tau_{2}-|\xi_{2}|^{2}|,|\tau_{3}-\sigma |\xi_{3}|^{2}|\}
\ll \max_{1\leq j\leq 3}|\xi_{j}|^{2},
\]
then we have
\begin{equation*}
|\xi_1| \sim |\xi_2| \gtrsim |\xi_3|.
\end{equation*}
\end{lemm}
\begin{proof}
We set
\[
\Phi (\xi_1, \xi_2, \xi_3):= |\xi_1|^2-|\xi_2|^2-\sigma |\xi_3|^2. 
\]
Then,  we have
\[
\begin{split}
\Phi(\xi_1,\xi_2,\xi_3) 
&=(1-\sigma)|\xi_1|^2-2\sigma\xi_1\cdot \xi_2 -(1+\sigma)|\xi_2|^2\\
&=(1-\sigma)\left|\xi_1-\frac{\sigma}{1-\sigma}\xi_2\right|^2-\frac{1}{1-\sigma}|\xi_2|^2\\
&=-(1+\sigma)\left|\xi_2+\frac{\sigma}{1+\sigma}\xi_1\right|^2+\frac{1}{1+\sigma}|\xi_1|^2
\end{split}
\]
and
\[
\begin{split}
\Phi(\xi_1,\xi_2,\xi_3) 
&=2\xi_2\cdot \xi_3+(1-\sigma)|\xi_3|^2\\
&=(1-\sigma)\left|\xi_3+\frac{1}{1-\sigma}\xi_2\right|^2-\frac{1}{1-\sigma}|\xi_2|^2.
\end{split}
\]
Therefore, if it holds $|\xi_1|\gg |\xi_2|$ or
$|\xi_1|\ll |\xi_2|$ or $|\xi_3|\gg |\xi_2|$, 
then we have 
\[
\max
\{|\tau_{1}+|\xi_{1}|^{2}|,|\tau_{2}-|\xi_{2}|^{2}|,|\tau_{3}-c|\xi_{3}|^{2}|\}
\gtrsim 
|\Phi(\xi_1,\xi_2,\xi_3)|\sim \max_{1\leq j\leq 3}|\xi_{j}|^{2}. 
\]
It implies the conclusion.
\end{proof}

Lemma \ref{modul_est_2} is different from the case $\kappa \neq 0$,
where $\kappa$ is as in \eqref{coeff}.
More precisely, $|\xi_2| \gg |\xi_3|$ occurs in this case,
while we have $|\xi_1| \sim |\xi_2| \sim |\xi_3|$ if $\kappa \neq 0$ (see Lemma 4.1 in \cite{Hi}).
\section{Proof of bilinear estimates for $d=2$\label{be_for_2d}}
%
%
In this section, 
we prove Proposition~\ref{key_be} for $d=2$.
To treat the low modulation interaction, 
we first introduce the angular frequency localization operators which were utilized in \cite{BHHT09}.
\begin{defn}[\cite{BHHT09}]
We define the angular decomposition of $\R^2$ in frequency.
We define a partition of unity in $\R$,
\begin{equation*}
1 = \sum_{j \in \Z} \omega_j, \qquad \omega_j (s) = \psi(s-j) \left( \sum_{k \in \Z} \psi (s-k) \right)^{-1}. 
\end{equation*}
For a dyadic number $A \geq 64$, we also define a partition of unity on the unit circle,
\begin{equation*}
1 = \sum_{j =0}^{A-1} \omega_j^A, \qquad \omega_j^A (\vartheta) = 
\omega_j \left( \frac{A\vartheta}{\pi} \right) + \omega_{j-A} \left( \frac{A\vartheta}{\pi} \right).
\end{equation*}
We observe that $\omega_j^A$ is supported in 
\begin{equation*}
\Theta_j^A = \left[\frac{\pi}{A} \, (j-2), \ \frac{\pi}{A} \, (j+2) \right] 
\cup \left[-\pi + \frac{\pi}{A} \, (j-2), \ - \pi +\frac{\pi}{A} \, (j+2) \right].
\end{equation*}
We now define the angular frequency localization operators $R_j^A$,
\begin{equation*}
\F_x (R_j^A f)(\xi) = \omega_j^A(\vartheta) \F_x f(\xi), \qquad \textnormal{where} \ \xi = |\xi| 
(\cos \vartheta, \sin \vartheta).
\end{equation*}
For any function $u  : \, \R \, \times \, \R^2 \, \to \C$,
we set 
$(R_j^A u ) (t, x) = (R_j^Au( t, \cdot)) (x)$. This operator localizes function in frequency to the set
\begin{equation}
{\mathfrak{D}}_j^A = \{ (\tau, |\xi| \cos \vartheta, |\xi| \sin \vartheta) \in \R \times \R^2 
\, | \, \vartheta \in \Theta_j^A  \} .
\label{DjA}
\end{equation}
Immediately, we can see
\begin{equation*}
u = \sum_{j=0}^{A-1} R_j^A u.
\end{equation*}
\end{defn}

The following propositions play an important role 
in the proof of Proposition~\ref{key_be}.

\begin{prop}\label{BSE_HHL}
Let $N_1$, $N_2$, $N_3$, $L_1$, $L_2$, $L_3$, $A\in 2^{\N_0}$, and
$j_1$, $j_2\in \{0,1,\dots ,A-1\}$.
We assume $A\ge 64$, $|j_1-j_2|\lesssim 1$, and $N_3\lesssim N_1\sim N_2$. 
Then,  
we have the following estimate:
\begin{align}
\begin{split}
&\|P_{N_3}(R_{j_1}^AQ_{L_1}^{1}P_{N_1}u_1\cdot R_{j_2}^AQ_{L_2}^{-1}P_{N_2}u_2)\|_{L^2_{tx}(\R\times \R^2)}\\
&\hspace{3ex}
\lesssim \left(\frac{N_1}{N_3A}\right)^{\frac{1}{2}}
L_1^{\frac{1}{2}}L_2^{\frac{1}{2}}
\|R_{j_1}^AQ_{L_1}^{1}P_{N_1}u_1\|_{L^2_{tx}(\R\times \R^2)}
\|R_{j_2}^AQ_{L_2}^{-1}P_{N_2}u_2\|_{L^2_{tx}(\R\times \R^2)}.
\end{split} \label{ABSE1}
\end{align}
\end{prop}
\begin{prop}\label{BSE_HHL_2}
Let $\sigma \in \R\backslash \{0,\pm 1\}$.
Let $N_1$, $N_2$, $N_3$, $L_1$, $L_2$, $L_3$, $A\in 2^{\N_0}$, and
$j_1$, $j_2\in \{0,1,\dots ,A-1\}$.
We assume $L_{\textnormal{max}} \ll N_{\max}^2$, $A\ge 64$, and $|j_1-j_2|\le 32$. 
Then,  
we have the following estimate:
\begin{align}
\begin{split}
&\|R_{j_1}^AQ_{L_1}^{-1}P_{N_1}(R_{j_2}^AQ_{L_2}^{-1}P_{N_2}u_2\cdot Q_{L_3}^{-\sigma}P_{N_3}u_3)\|_{L^2_{tx}(\R\times \R^2)}\\
&\hspace{3ex}
\lesssim A^{-\frac{1}{2}}
L_2^{\frac{1}{2}}L_3^{\frac{1}{2}}
\|R_{j_2}^AQ_{L_2}^{-1}P_{N_2}u_2\|_{L^2_{tx}(\R\times \R^2)}
\|Q_{L_3}^{-\sigma}P_{N_3}u_3\|_{L^2_{tx}(\R\times \R^2)}.
\end{split}\label{ABSE2}
\end{align}
\end{prop}
\begin{prop}\label{thm2.6}
Let $\sigma \in \R\backslash \{0,\pm 1\}$.
Let $N_1$, $N_2$, $N_3$, $L_1$, $L_2$, $L_3$, $A\in 2^{\N_0}$, 
and $j_1$, $j_2\in \{0,1,\dots ,A-1\}$. 
We assume $L_{\textnormal{max}} \ll N_{\max}^2$, 
$A\ge 64$, and 
$16 \leq |j_1 - j_2 |\leq 32$. 
Then the following estimate holds:
\begin{equation}
\begin{split}
\|Q_{L_3}^{\sigma} P_{N_3}(R_{j_1}^A Q_{L_1}^{1}P_{N_1}u_{1}\cdot 
R_{j_2}^A Q_{L_2}^{-1}P_{N_2}u_{2})\|_{L^{2}_{tx}(\R\times \R^2)} & \\
\lesssim A^{\frac{1}{2}} N_{\max}^{-1}L_1^{\frac{1}{2}}L_2^{\frac{1}{2}} L_3^{\frac{1}{2}} 
\|R_{j_1}^A Q_{L_1}^{1}P_{N_1}u_{1}\|_{L^2_{tx}(\R\times \R^2)} &  \|R_{j_2}^A Q_{L_2}^{-1}P_{N_2}u_{2}\|_{L^2_{tx}(\R\times \R^2)}.
\end{split}\label{0609}
\end{equation}
\end{prop}
The bound such as in \eqref{ABSE1} does not appear for $\kappa \neq 0$ in \cite{Hi} and \cite{HK},
where $\kappa$ is as in \eqref{coeff}.
On the other hand,
similar bounds as in Propositions~\ref{BSE_HHL_2} and ~\ref{thm2.6}
also appear for $\kappa \neq 0$.
See Theorem 2.8 and Proposition 2.9 in \cite{HK}. 
However, we need to treat the case $N_3 \ll N_1 \sim N_3$ in Propositions~\ref{BSE_HHL_2} and ~\ref{thm2.6}.
Hence, while a part of proof of Propositions~\ref{BSE_HHL_2} and ~\ref{thm2.6} is similar to that in the previous results,
we need a more careful calculation.
See Lemma \ref{angle_est_A} below for example.

We postpone the proof of Propositions \ref{BSE_HHL}, \ref{BSE_HHL_2}, and \ref{thm2.6} to the next subsection.
Assuming Propositions \ref{BSE_HHL}, \ref{BSE_HHL_2}, and \ref{thm2.6},
we here prove the bilinear estimates.
\begin{proof}[Proof of Proposition~\ref{key_be} for $d=2$]
Let $s\ge \frac{1}{2}$ and
\[
(s_1,s_2,s_3)\in \{(s,s,-s),(s,-s,s),(-s,s,s)\}.
\]
We prove (\ref{key_est}). 
We set
\[
u_{N_1,L_1}:=Q_{L_1}^{1}P_{N_1}U,\ 
v_{N_2,L_2}:=Q_{L_2}^{-1}P_{N_2}V,\ 
w_{N_3,L_3}:=Q_{L_3}^{-\sigma}P_{N_3}W. 
\]
Then, we have
\[
\begin{split}
&
\left|\int_{\R \times \R^2} U(t,x)V(t,x)\p_jW(t,x)dxdt\right|\\
&\lesssim \sum_{N_1,N_2,N_3\ge 1}\sum_{L_1,L_2, L_3 \ge 1} N_{3}
\left|\int_{\R \times \R^2} u_{N_1,L_1}v_{N_2,L_2}w_{N_3,L_3}dxdt\right|. 
\end{split}
\]
It suffices to show that
\begin{equation}\label{desired_est}
\begin{split}
&N_{3}
\left|\int_{\R \times \R^2} u_{N_1,L_1}v_{N_2,L_2}w_{N_3,L_3}dxdt\right|\\
&\lesssim
N_{\min}^{s}(L_1L_2L_3)^{c}\left(\frac{N_{\min}}{N_{\max}}\right)^{\eps}
\|u_{N_1,L_1}\|_{L^2_{tx}}\|v_{N_2,L_2}\|_{L^2_{tx}}\|w_{N_3,L_3}\|_{L^2_{tx}}
\end{split}
\end{equation}
for some $b'\in (0,\frac{1}{2})$, $c\in (0,b')$, and $\eps >0$. 
Indeed, from (\ref{desired_est}) and the Cauchy-Schwarz inequality, 
we obtain 
\[
\begin{split}
&\sum_{\substack{N_1,N_2,N_3 \ge 1 \\ N_{1} \lesssim N_{2} \sim N_{3}}}
\sum_{L_1,L_2, L_3 \ge 1} N_{3}
\left|\int_{\R \times \R^2} u_{N_1,L_1}v_{N_2,L_2}w_{N_3,L_3}dxdt\right|\\
&\lesssim 
 \sum_{\substack{N_1,N_2,N_3 \ge 1 \\ N_1 \lesssim N_2 \sim N_3}}
 \sum_{L_1,L_2, L_3\ge 1}
N_{1}^{s}\left(\frac{N_{\min}}{N_{\max}}\right)^{\eps}
(L_1L_2L_3)^{c}\|u_{N_1,L_1}\|_{L^2_{tx}}\|v_{N_2,L_2}\|_{L^2_{tx}}\|w_{N_3,L_3}\|_{L^2_{tx}}\\
&\lesssim  \sum_{N_3 \ge 1}\sum_{\substack{N_2 \ge 1 \\ N_2 \sim N_3}}
N_2^{-(s_2+s_3)-\eps}
\bigg( \sum_{\substack{N_1 \ge 1 \\ N_1\lesssim N_2}}
N_1^{s-s_1+\eps}
N_1^{s_1} \sum_{L_1 \ge 1}L_1^{-(b'-c)}L_1^{b'}\|u_{N_1,L_1}\|_{L^2_{tx}} \bigg) 
\\
& \quad \times \bigg(N_2^{s_2}\sum_{L_2\ge 1}L_2^{-(b'-c)}L_2^{b'}\|v_{N_2,L_2}\|_{L^2_{tx}}\bigg)
\bigg(N_3^{s_3}\sum_{L_3\ge 1}L_3^{-(b'-c)}L_3^{b'}\|w_{N_3,L_3}\|_{L^2_{tx}} \bigg)\\
& \lesssim \|u\|_{X^{s_1,b'}_{1}} \|v\|_{X^{s_2,b'}_{-1}} \|w\|_{X^{s_3, b'}_{-\sigma}},
\end{split}
\]
since $s-s_1\ge 0$, $s-s_1-s_2-s_3 = 0$, and $b'-c>0$. 
The summations for $N_2\lesssim N_1\sim N_3$ and 
$N_3\lesssim N_1\sim N_2$ are similarly handled.

Now, we prove (\ref{desired_est}).\\
\kuuhaku \\
\underline{Case\ 1:\ High modulation, $\displaystyle L_{\max}\gtrsim N_{\max}^2N_{\min}^{-\frac{2}{3}}$}

We first assume $L_3=L_{\max}$. 
By the symmetry, we can assume $N_1\le N_2$.
Then, by the Cauchy-Schwarz inequality and (\ref{L2be_est_2}), 
we have
\[
\begin{split}
&N_3\left|\int_{\R \times \R^2}u_{N_1,L_1}v_{N_2,L_2}w_{N_3,L_3}dxdt\right|\\
&\lesssim N_{3}\|P_{N_3}(u_{N_1,L_1}v_{N_2,L_2})\|_{L^2_{tx}}\|w_{N_3,L_3}\|_{L^2_{tx}}\\
&\lesssim N_{3}N_{1}^{4\delta}
\left(\frac{N_{1}}{N_{2}}\right)^{\frac{1}{2}- 2\delta}L_2^{c}L_3^{c}
\|u_{N_1,L_1}\|_{L^2_{tx}}\|v_{N_2,L_2}\|_{L^2_{tx}}\|w_{N_3,L_3}\|_{L^2_{tx}}\\
&\lesssim N_{3}N_{1}^{4\delta}
\left(\frac{N_{1}}{N_{2}}\right)^{\frac{1}{2}- 2\delta}
(N_{\max}^2N_{\min}^{-\frac{2}{3}})^{-c}(L_1L_2L_3)^{c}\\
&\hspace{15ex}\times \|u_{N_1,L_1}\|_{L^2_{tx}}\|v_{N_2,L_2}\|_{L^2_{tx}}\|w_{N_3,L_3}\|_{L^2_{tx}}
,
\end{split}
\]
where $\delta = \frac{1}{2}-c$. 
If $N_3\lesssim N_1\sim N_2$, then we obtain
\[
N_{3}N_{1}^{4\delta}
\left(\frac{N_{1}}{N_{2}}\right)^{\frac{1}{2}- 2\delta}
(N_{\max}^2N_{\min}^{-\frac{2}{3}})^{-c}
\sim N_3^{3-\frac{16}{3}c-s}N_3^s\left(\frac{N_3}{N_1}\right)^{6c-2}. 
\]
If $N_1\lesssim N_2\sim N_3$, then we obtain
\[
N_{3}N_{1}^{4\delta}
\left(\frac{N_{1}}{N_{2}}\right)^{\frac{1}{2}- 2\delta}
(N_{\max}^2N_{\min}^{-\frac{2}{3}})^{-c}
\sim N_1^{3-\frac{16}{3}c-s}N_1^s\left(\frac{N_1}{N_3}\right)^{4c-\frac{3}{2}}.
\]
Therefore, by choosing $b'$ and $c$ as
$\max\{\frac{3(3-s)}{16},\frac{3}{8},\frac{1}{3}\}<c<b'<\frac{1}{2}$
for $s> \frac{1}{3}$, 
we get (\ref{desired_est}). 
The case $L_1=L_{\max}$ and $L_2=L_{\max}$ is similarly treated, 
but we use (\ref{L2be_est_20}) 
instead of  (\ref{L2be_est_2}). \\
\kuuhaku \\
\underline{Case\ 2:\ Low modulation, $\displaystyle L_{\max}\ll N_{\max}^2N_{\min}^{-\frac{2}{3}}$\ $(\lesssim N_{\max}^2)$}

By Lemma \ref{modul_est_2}, 
we can assume $N_3 \lesssim N_1 \sim N_2$. 
We set
\begin{equation}
M :=L_{\max}^{-\frac{3}{4}}N_1^{\frac{3}{2}}N_3^{-\frac{1}{2}} \gg 1
\label{mod2}
\end{equation}
and decompose $\R^3 \cross \R^3$ as follows:
\begin{equation*}
\R^3 \cross \R^3 =
\bigg( \bigcup_{\tiny{\substack{0 \leq j_1,j_2 \leq M -1\\|j_1 - j_2|\leq 16}}} 
{\mathfrak{D}}_{j_1}^{M} \cross {\mathfrak{D}}_{j_2}^{M} \bigg)
 \cup 
\bigg(\bigcup_{64 \leq A \leq M} \ \bigcup_{\tiny{\substack{0 \leq j_1,j_2 \leq A -1\\ 16 \leq |j_1 - j_2|\leq 32}}} 
{\mathfrak{D}}_{j_1}^A \cross {\mathfrak{D}}_{j_2}^A\bigg),
\end{equation*}
where $\mathfrak{D}_j^A$ is as in \eqref{DjA}.
We can write
\begin{align}
 &\left|\int_{\R \times \R^2}u_{N_1,L_1}v_{N_2,L_2}w_{N_3,L_3}dxdt\right| \notag\\
 &\leq \sum_{{\tiny{\substack{0 \leq j_1,j_2 \leq M -1\\|j_1 - j_2|\leq 16}}}} 
\left|\int_{\R \times \R^2}
u_{N_1,L_1, j_1}^M v_{N_2,L_2, j_2}^M w_{N_3,L_3}dxdt\right| \notag\\ 
&\qquad
+
\sum_{64 \leq A \leq M}  \sum_{{\tiny{\substack{0 \leq j_1,j_2 \leq A-1\\ 16 \le |j_1 - j_2|\leq 32}}}} 
\left|\int_{\R \times \R^2}
u_{N_1,L_1, j_1}^A v_{N_2,L_2, j_2}^A w_{N_3,L_3}dxdt\right| \notag\\
&=: {\rm I} + {\rm \II},
\label{I,II}
\end{align}
where
\[
u_{N_1,L_1,j_1}^A=R^A_{j_1}u_{N_1,L_1},\ \ 
v_{N_2,L_2,j_2}^A=R^A_{j_2}v_{N_2,L_2}. 
\]
For the contribution from the first term $\rm I$ in \eqref{I,II}, 
we first assume $L_{\textnormal{max}} = L_3$. 
By \eqref{I,II}, the H\"older inequality, (\ref{ABSE1}), and \eqref{mod2},
we get
\begin{align*}
N_3 \cdot {\rm I}
& \lesssim  
\sum_{{\tiny{\substack{0 \leq j_1,j_2 \leq M-1\\|j_1 - j_2|\leq 16}}}}
N_3\|P_{N_3} ( u_{N_1,L_1, j_1}^M v_{N_2,L_2, j_2}^M)\|_{L^{2}_{tx}}
\|w_{N_3,L_3} \|_{{L^2_{t x}}}  \\
& \lesssim 
\sum_{{\tiny{\substack{0 \leq j_1,j_2 \leq M-1\\|j_1 - j_2|\leq 16}}}}N_3
\left(\frac{N_1}{N_3M}\right)^{\frac{1}{2}} L_1^{\frac{1}{2}}L_2^{\frac{1}{2}} 
\|u_{N_1,L_1, j_1}^M\|_{L^2_{tx}}   \|v_{N_2,L_2, j_2}^M\|_{L^2_{tx}}
\|w_{N_3,L_3} \|_{{L^2_{t x}}}\\
& \lesssim  N_3^{\frac{1}{2}}\left(\frac{N_3}{N_1}\right)^{\frac{1}{4}}
  L_1^{\frac{1}{2}}L_2^{\frac{1}{2}}
L_3^{\frac{3}{8}} \|u_{N_1,L_1}\|_{L^2_{tx}}\|v_{N_2,L_2}\|_{L^2_{tx}}\|w_{N_3,L_3}\|_{L^2_{tx}}\\
&\lesssim N_3^{\frac{1}{2}}\left(\frac{N_3}{N_1}\right)^{\frac{1}{4}}  (L_1L_2L_3)^{\frac{11}{24}} \|u_{N_1,L_1}\|_{L^2_{tx}}\|v_{N_2,L_2}\|_{L^2_{tx}}\|w_{N_3,L_3}\|_{L^2_{tx}},
\end{align*}
which shows \eqref{desired_est}.
If $L_{\max}=L_1$ or $L_{\max}=L_2$, 
then we can use a better estimate (\ref{ABSE2}) instead of (\ref{ABSE1}).
Hence, we obtain \eqref{desired_est} in this case.

For the second term $\rm \II$ in \eqref{I,II}, by Proposition \ref{thm2.6} and \eqref{mod2},
we get
\begin{align*}
N_3 \cdot {\rm \II}
& \lesssim  
\sum_{64 \leq A \leq M} \sum_{{\tiny{\substack{0 \leq j_1,j_2 \leq A-1\\ 16 \le |j_1 - j_2|\leq 32}}}} 
N_3\|Q_{L_3}^{\sigma} P_{N_3} ( u_{N_1,L_1, j_1}^A v_{N_2,L_2, j_2}^A)\|_{L^{2}_{tx}} 
\|w_{N_3,L_3} \|_{{L^2_{t x}}}\\
& \lesssim 
\sum_{64 \leq A \leq M}  
\sum_{{\tiny{\substack{0 \leq j_1,j_2 \leq A-1\\
16 \le |j_1 - j_2|\leq 32}}}}  
 N_3A^{\frac{1}{2}}N_1^{-1} ( L_1 L_2 L_3)^\frac{1}{2} \\
&\hspace{80pt}
\times
\|u_{N_1,L_1, j_1}^A\|_{L^2_{tx}}   \|v_{N_2,L_2, j_2}^A\|_{L^2_{tx}}
\|w_{N_3,L_3} \|_{{L^2_{t x}}}\\
&  \lesssim N_3^{\frac{1}{2}}\left(\frac{N_3}{N_1}\right)^{\frac{1}{4}}
L_{\max}^{-\frac{3}{8}} ( L_1 L_2 L_3)^{\frac{1}{2}} \|u_{N_1,L_1}\|_{L^2_{tx}}\|v_{N_2,L_2}\|_{L^2_{tx}}\|w_{N_3,L_3}\|_{L^2_{tx}}\\
&\lesssim N_3^{\frac{1}{2}}\left(\frac{N_3}{N_1}\right)^{\frac{1}{4}} ( L_1 L_2 L_3)^{\frac{3}{8}} \|u_{N_1,L_1}\|_{L^2_{tx}}\|v_{N_2,L_2}\|_{L^2_{tx}}\|w_{N_3,L_3}\|_{L^2_{tx}},
\end{align*} 
which shows \eqref{desired_est}.
This completes the proof of Proposition \ref{key_be} for $d=2$.
\end{proof}

\subsection{Proof of key propositions}

In this subsection, we prove Propositions \ref{BSE_HHL}, \ref{BSE_HHL_2}, and \ref{thm2.6}.
First, we show Proposition \ref{BSE_HHL}.

\begin{proof}[Proof of Proposition~\ref{BSE_HHL}]
If $A\lesssim \frac{N_1}{N_3}$, then we obtain (\ref{ABSE1}) 
by the H\"older inequality and Corollary~\ref{Bo_Stri} with $p=q=4$,
since $1\lesssim \frac{N_1}{N_3A}$. 
Therefore, we can assume
\begin{equation}
A\gg \frac{N_1}{N_3} \gtrsim 1.
\label{AA}
\end{equation} 
We set $f_1=\F[R^A_{j_1}Q_{L_1}^{1}P_{N_1}u_1]$ and 
$f_2=\F[R^A_{j_2}Q_{L_2}^{-1}P_{N_2}u_2]$. 
By the duality argument, it suffices to show that
\begin{equation}\label{BSE_pf_1}
\begin{split}
&\left|\int_{\Omega}f_1(\tau_1,\xi_1)f_2(\tau_2,\xi_2)f(\tau_1+\tau_2,\xi_1+\xi_2)d\tau_1d\tau_2d\xi_1d\xi_2\right|\\
&\lesssim \left(\frac{N_1}{N_3A}\right)^{\frac{1}{2}}L_1^{\frac{1}{2}}L_2^{\frac{1}{2}}\|f_1\|_{L^2_{\tau \xi}}\|f_2\|_{L^2_{\tau \xi}}\|f\|_{L^2_{\tau \xi}}
\end{split}
\end{equation}
for any $f\in L^2(\R\times \R^2)$,
where $\xi_i =(\xi_i^{(1)}, \xi_i^{(2)})=|\xi_i|(\cos \theta_i, \sin \theta_i),\ (i=1,2)$ and
\[
\Omega =\left\{(\tau_1,\tau_2,\xi_1,\xi_2)
\left|
\begin{split}
&\ |\xi_1|\sim N_1,\ |\xi_2|\sim N_2,\ |\xi_1+\xi_2|\sim N_3,\\
&\ \theta_1\in \Theta_{j_1}^A,\ \theta_2\in \Theta_{j_2}^A,\\
&\ |\tau_1+|\xi_1|^2|\sim L_1,\ |\tau_2-|\xi_2|^2|\sim L_2
\end{split}
\right.
\right\}.
\]
By the Cauchy-Schwarz inequality, we have
\begin{equation}\label{BSE_pf_2}
\begin{split}
&\left|\int_{\Omega}f_1(\tau_1,\xi_1)f_2(\tau_2,\xi_2)f(\tau_1+\tau_2,\xi_1+\xi_2)d\tau_1d\tau_2d\xi_1d\xi_2\right|\\
&\lesssim \|f_1\|_{L^2_{\tau\xi}}\|f_2\|_{L^2_{\tau\xi}}
\left(\int_{\Omega}|f(\tau_1+\tau_2,\xi_1+\xi_2)|^2d\tau_1d\tau_2d\xi_1d\xi_2\right)^{\frac{1}{2}}.
\end{split}
\end{equation}
By changing variables $(\xi_1,\xi_2)\mapsto (\widetilde{\xi}_1, \widetilde{\xi}_2)$ as 
\[
(\widetilde{\xi}_1, \widetilde{\xi}_2)
=\left(\xi_1R\left(-\frac{\pi}{A}j_2\right), \xi_2R\left(-\frac{\pi}{A}j_2\right)\right),\ 
R(\theta )=
\begin{pmatrix}
\cos \theta & \sin \theta \\
-\sin \theta & \cos \theta 
\end{pmatrix},
\]
we have
\begin{equation}\label{BSE_pf_3}
\begin{split}
&\int_{\Omega}|f(\tau_1+\tau_2,\xi_1+\xi_2)|^2d\tau_1d\tau_2d\xi_1d\xi_2\\
&=\int_{\widetilde{\Omega}}\left|f\left(\tau_1+\tau_2,(\widetilde{\xi}_1+\widetilde{\xi}_2)R\left(\frac{\pi}{A}j_2\right)\right)\right|^2d\tau_1d\tau_2d\widetilde{\xi}_1d\widetilde{\xi}_2,
\end{split}
\end{equation}
where $\widetilde{\xi}_i =(\widetilde{\xi}_i^{(1)}, \widetilde{\xi}_i^{(2)})=|\widetilde{\xi}_i|(\cos \widetilde{\theta}_i, \sin \widetilde{\theta}_i)$, 
$(i=1,2)$ and
\[
\widetilde{\Omega} =\left\{(\tau_1,\tau_2,\widetilde{\xi}_1,\widetilde{\xi}_2)
\left|
\begin{split}
&\ |\widetilde{\xi}_1|\sim N_1,\ |\widetilde{\xi}_2|\sim N_2,\ 
|\widetilde{\xi}_1+\widetilde{\xi}_2|\sim N_3,\\
&\ \widetilde{\theta}_1\in \Theta_{j_1-j_2}^A,\ 
\widetilde{\theta}_2\in \Theta_{0}^A,\\
&\ |\tau_1+|\widetilde{\xi}_1|^2|\sim L_1,\ 
|\tau_2-|\widetilde{\xi}_2|^2|\sim L_2
\end{split}
\right.
\right\}.
\]
Since $|j_1-j_2|\lesssim 1$, we have
\[
\min\{|\widetilde{\theta}_1|, |\pi -\widetilde{\theta}_1|\}\lesssim A^{-1},\ \ 
\min\{|\widetilde{\theta}_2|, |\pi -\widetilde{\theta}_2|\}\lesssim A^{-1}
\]
for $\widetilde{\xi}_1\in \Theta_{j_1-j_2}^A$ and $\widetilde{\xi}_2\in \Theta_0^A$. 
Therefore, it follows from \eqref{AA} that
\[
|\widetilde{\xi}_1^{(2)}+\widetilde{\xi}_2^{(2)}|
\le |\widetilde{\xi}_1||\sin \widetilde{\theta}_1|
+|\widetilde{\xi}_2||\sin \widetilde{\theta}_2|
\lesssim (N_1+N_2)A^{-1}\ll N_3. 
\]
It says that $|\widetilde{\xi}_1^{(1)}+\widetilde{\xi}_2^{(1)}|\sim N_3$,
since $|\widetilde{\xi}_1+\widetilde{\xi}_2|\sim N_3$ in $\widetilde{\Omega}$. 
By changing variables $(\tau_1,\tau_2)\mapsto (c_1,c_2)$ 
and $(\widetilde{\xi}_1,\widetilde{\xi}_2)\mapsto (\mu,w,z)$ as 
\[
\begin{split}
&c_1=\tau_1+|\widetilde{\xi}_1|^2,\ \  
c_2=\tau_2-|\widetilde{\xi}_2|^2,\\
&\mu  =c_1+c_2-|\widetilde{\xi}_1|^2+|\widetilde{\xi}_2|^2,\\  
&w=\widetilde{\xi}_1+\widetilde{\xi}_2,\ z=\widetilde{\xi}_2^{(2)},
\end{split}
\]
we have
\[
\begin{split}
&\int_{\widetilde{\Omega}}\left|f\left(\tau_1+\tau_2,(\widetilde{\xi}_1+\widetilde{\xi}_2)R\left(\frac{\pi}{A}j_2\right)\right)\right|^2d\tau_1d\tau_2d\widetilde{\xi}_1d\widetilde{\xi}_2\\
&\lesssim \bigg(\int_{\substack{|c_1|\sim L_1\\ |c_2|\sim L_2}}
dc_1dc_2\bigg)
\left(\int_{|z|\lesssim N_2A^{-1}}dz\right)\\
&\hspace{10ex}\times \left(\int_{\R \times \R^2} \left|f\left(\mu,wR\left(\frac{\pi}{A}j_2\right)\right)\right|^2
\ee_{\{|\widetilde{\xi}_1+\widetilde{\xi}_2|\sim N_3\}}(\widetilde{\xi}_1,\widetilde{\xi}_2)J(\widetilde{\xi}_1,\widetilde{\xi}_2)^{-1}d\mu dw\right),
\end{split}
\]
where
\[
J(\widetilde{\xi}_1,\widetilde{\xi}_2)
=\left|{\rm det}\frac{\partial (\mu ,w,z)}{\partial (\widetilde{\xi}_1,\widetilde{\xi}_2)}\right|
=|\widetilde{\xi}_1^{(1)}+\widetilde{\xi}_2^{(1)}|\sim N_3. 
\]
Therefore, we obtain
\begin{equation}\label{BSE_pf_4}
\int_{\widetilde{\Omega}}\left|f\left(\tau_1+\tau_2,(\widetilde{\xi}_1+\widetilde{\xi}_2)R\left(\frac{\pi}{A}j_2\right)\right)\right|^2d\tau_1d\tau_2d\widetilde{\xi}_1d\widetilde{\xi}_2
\lesssim \frac{N_2A^{-1}}{N_3}L_1L_2\|f\|_{L^2}^2
\end{equation}
since
\[
\int_{\R \times \R^2} \left|f\left(\mu,wR\left(\frac{\pi}{A}j_2\right)\right)\right|^2d\mu dw
=\int_{\R \times \R^2} |f(\mu,w)|^2d\mu dw=\|f\|_{L^2}^2.
\]
As a result, we get (\ref{BSE_pf_1}) from (\ref{BSE_pf_2}), (\ref{BSE_pf_3}), 
and (\ref{BSE_pf_4}).
\end{proof}

Before the proof of Proposition \ref{BSE_HHL_2},
we state an elementary lemma.

\begin{lemm} \label{element1}
Let $L$, $M$, $N \in 2^{\N_0}$.
Assume that $N^2 \gg L$.
Then, we have
\[
| \{ x \in \R \ | \ |(x-M)^2 -N^2| \lesssim L \}|
\lesssim \frac{L}{N}.
\]
\end{lemm}

\begin{proof}
When $|(x-M)^2 -N^2| \lesssim L$,
a direct calculation shows that $x$ is in
\[
\left[M-\sqrt{N^2-C L}, M-\sqrt{N^2+CL}\right]
\cup
\left[M+\sqrt{N^2-CL},M+\sqrt{N^2+CL}\right]
\]
for some constant $C>0$.
Hence,
from $N^2 \gg L$,
we have
\begin{align*}
| \{ x \in \R \ | \ |(x-M)^2 -N^2| \lesssim L \}|
&\lesssim \sqrt{N^2+CL} - \sqrt{N^2-CL} \\
&= \frac{2CL}{\sqrt{N^2+CL} + \sqrt{N^2-CL}} \\
&\sim \frac{L}{N},
\end{align*}
which concludes the proof.
\end{proof}

We are now in position to prove Proposition~\ref{BSE_HHL_2}.

\begin{proof}[Proof of Proposition~\ref{BSE_HHL_2}]
By Lemma~\ref{modul_est_2} and $L_{\max}\ll N_{\max}^2$, 
we can assume $N_3\lesssim N_1\sim N_2$. 
By Plancherel's theorem as in Remark \ref{duality}, it suffices to show that
\begin{equation}
\begin{split}
&\left\|\psi_{N_1,L_1,j_1}(\tau_1,\xi_1)
\int_{\R \times \R^2} f_2(\tau_2,\xi_2)f_3(\tau_1+\tau_2,\xi_1+\xi_2)d\tau_2d\xi_2
\right\|_{L^2_{\tau_1 \xi_1}}\\
&\lesssim A^{-\frac{1}{2}}L_2^{\frac{1}{2}}L_3^{\frac{1}{2}}
\|f_2\|_{L^2_{\tau \xi}}\|f_3\|_{L^2_{\tau\xi}},
\end{split}
\label{AP0b}
\end{equation}
where  $f_2=\F[R_{j_2}^AQ_{L_2}^{-1}P_{N_1}u_2]$, $f_3=\F[Q_{L_3}^{\sigma}P_{N_3}u_3]$, and
$\psi_{N_1,L_1,j_1}(\tau_1,\xi_1) =\omega_{j_1}^{A}(\theta_1)
\psi_{L_1}(\tau_1+|\xi_1|^2)\psi_{N_1}(\xi_1)$
for $(\tau_1,\xi_1)=(\tau_1, |\xi_1|\cos \theta_1, |\xi_1|\sin \theta_1)$.
If $A \sim 1$, \eqref{AP0b} follows from Corollary~\ref{Bo_Stri} with $p=q=4$.
We hence assume that $A \gg 1$.

By the Cauchy-Schwarz inequality, we have
\begin{equation}
\begin{split}
&\left\|\psi_{N_1,L_1,j_1}(\tau_1,\xi_1)
\int_{\R \times \R^2} f_2(\tau_2,\xi_2)f_3(\tau_1+\tau_2,\xi_1+\xi_2)d\tau_2d\xi_2
\right\|_{L^2_{\tau_1 \xi_1}}\\
&\lesssim 
\bigg\|\psi_{N_1,L_1,j_1}(\tau_1,\xi_1) \\
&\qquad
\times 
\left(\int_{\R \times \R^2} |f_2(\tau_2,\xi_2)|^2|f_3(\tau_1+\tau_2,\xi_1+\xi_2)|^2d\tau_2d\xi_2
\right)^{\frac{1}{2}} |E(\tau_1,\xi_1)|^{\frac{1}{2}}
\bigg\|_{L^2_{\tau_1 \xi_1}}\\
&\lesssim \sup_{(\tau_1,\xi_1)\in {\rm supp}\psi_{N_1,L_1,j_1}}
|E(\tau_1,\xi_1)|^{\frac{1}{2}}
\cdot \|f_2\|_{L^2_{\tau\xi}}\|f_3\|_{L^2_{\tau\xi}},
\end{split}
\label{AP0}
\end{equation}
where
\[
E(\tau_1,\xi_1)=\left\{(\tau_2,\xi_2)\in {\mathfrak{D}}_{j_2}^A\left|
\begin{split}
&\langle \tau_2-|\xi_2|^2\rangle\sim L_2,\ 
\langle \tau_1+\tau_2+\sigma |\xi_1+\xi_2|^2\rangle\sim L_3,\\
&|\xi_2|\sim N_2,\ |\xi_1+\xi_2|\sim N_3
\end{split}
\right.
\right\}.
\]
We set $\widetilde{E}(\tau_1,\xi_1)=\{\xi_2| \ (\tau_2,\xi_2)\in E(\tau_1,\xi_1) \text{ for some $\tau_2 \in \R$} \}$. 
Then, it holds that
\begin{equation}
|E(\tau_1,\xi_1)|
\lesssim \min\{L_2,L_3\}
|\widetilde{E}(\tau_1,\xi_1)|.
\label{AP0a}
\end{equation}
For 
$\xi_2=(|\xi_2|\cos \theta_2,|\xi_2|\sin \theta_2)\in \widetilde{E}(\tau_1,\xi_1)$, we obtain
\begin{equation}
\begin{split}
&-(\tau_1+|\xi_1|^2) - (\tau_2-|\xi_2|^2)+(\tau_1+\tau_2+\sigma |\xi_1+\xi_2|^2)\\
&=-|\xi_1|^2+|\xi_2|^2+\sigma|\xi_1+\xi_2|^2\\
&=\frac{((1+\sigma)|\xi_2|+\sigma|\xi_1|\cos \angle (\xi_1,\xi_2))^2
-(1-\sigma^2\sin^2\angle (\xi_1,\xi_2))|\xi_1|^2}{1+\sigma}. 
\end{split}
\notag
\end{equation}
From $(\tau_1,\xi_1)\in {\rm supp}\psi_{N_1,L_1,j_1}$ and $\xi_2 \in \widetilde{E}(\tau_1,\xi_1)$,
it says that
\begin{equation}
\begin{split}
&((1+\sigma)|\xi_2|+\sigma|\xi_1|\cos \angle (\xi_1,\xi_2))^2 \\
&=
(1-\sigma^2\sin^2\angle (\xi_1,\xi_2))|\xi_1|^2
-(1+\sigma) (\tau_1+|\xi_1|^2)
+O(\max\{L_2,L_3\}). 
\end{split}
\label{AP1}
\end{equation}
Here, a simple calculation shows that
\begin{equation}
\frac{x}{2} \le \sin x \le x
\label{AP1a}
\end{equation}
for $0 \le x \le \frac{\pi}{2}$.
It follows from $\theta_1\in \Theta_{j_1}^A$, $\theta_2\in \Theta_{j_2}^A$,
$|j_1-j_2| \le 32$, and $A \gg 1$ that
\begin{equation}
|\sigma \sin \angle (\xi_1,\xi_2)|
\le |\sigma \angle (\xi_1,\xi_2)|
\lesssim A^{-1}
\ll 1.
\label{AP2}
\end{equation}
Since $\xi_2 =(|\xi_2| \cos \theta_2, |\xi_2| \sin \theta_2)$,
we may regard that $|\xi_2|$ and $\theta_2$ are independent variables.
Then, Lemma \ref{element1} with \eqref{AP1} and \eqref{AP2} yields that
$|\xi_2|$ is restricted to a set $\Omega$ of measure 
at most $O\left(\frac{\max\{L_2,L_3\}}{N_1}\right)$.
As a result, we obtain
\begin{equation}
\begin{split}
|\widetilde{E}(\tau,\xi)|
&\lesssim \int_{\theta_2\in \Theta_{j_2}^A}
\left(\int_{|\xi_2|\in \Omega}|\xi_2|d|\xi_2|\right)d\theta_2\\
&\lesssim A^{-1}\max\{L_2,L_3\},
\end{split}
\label{AP3}
\end{equation}
since $|\xi_2|\sim N_2\sim N_1$. 

Therefore, \eqref{AP0b} follows from \eqref{AP0}, \eqref{AP0a}, and \eqref{AP3}.
This concludes the proof of Proposition \ref{BSE_HHL_2}.
\end{proof}

To prove Proposition~\ref{thm2.6}, 
we use a nonlinear version of the Loomis-Whitney inequality. 
\begin{prop}[\cite{BHT10} Corollary 1.5] \label{prop2.7}
Assume that the surface $S_i$ $(i=1,2,3)$ 
is an open and bounded subset of $S_i^*$ which 
satisfies the following conditions \textnormal{(Assumption 1.1 in \cite{BHT10})}.

\textnormal{(i)} $S_i^*$ is defined as
\begin{equation*}
S_i^* = \{ {\lambda_i} \in U_i \ | \ \Phi_i({\lambda_i}) = 0 , \nabla \Phi_i \not= 0 \}
\end{equation*}
for a convex $U_i \subset \R^3$ such that $\dist (S_i, U_i^c) \geq$ \textnormal{diam}$(S_i)$ and $\Phi_i \in C^{1,1} (U_i)$;

\textnormal{(ii)} the unit normal vector field $\mathfrak{n}_i$ on $S_i^*$ satisfies the H\"{o}lder condition
\begin{equation*}
\sup_{\lambda, \lambda' \in S_i^*} \bigg(\frac{|\mathfrak{n}_i(\lambda) - 
\mathfrak{n}_i(\lambda')|}{|\lambda - \lambda'|}
+ \frac{|\mathfrak{n}_i(\lambda) \cdot ({\lambda} - {\lambda}')|}{|{\lambda} - {\lambda}'|^2} \bigg) \lesssim 1;
\end{equation*}

\textnormal{(iii)} there exists $a>0$ such that the matrix 
\[
{N}({\lambda_1}, {\lambda_2}, {\lambda_3}) = ({\mathfrak{n}_1} 
({\lambda_1}), {\mathfrak{n}_2}({\lambda_2}), {\mathfrak{n}_3}({\lambda_3}))
\]
satisfies the transversality condition
\begin{equation*}
a \leq \textnormal{det} {N}({\lambda_1}, {\lambda_2}, {\lambda_3})  \leq 1
\end{equation*}
for all $({\lambda_1}, {\lambda_2}, {\lambda_3}) \in {S_1^*} \cross {S_2^*} \cross {S_3^*}$.

We also assume \textnormal{diam}$({S_i}) \lesssim a$. 
Then for functions $f \in L^2 (S_1)$ and $g \in L^2 (S_2)$, the restriction of the convolution $f*g$ to 
$S_3$ is a well-defined $L^2(S_3)$-function which satisfies 
\begin{equation*}
\| f *g \|_{L^2(S_3)} \lesssim \frac{1}{\sqrt{a}} \| f \|_{L^2(S_1)} \| g\|_{L^2(S_2)}.
\end{equation*}
\end{prop}
We first claim the following. 
\begin{lemm}\label{angle_est_A}
Let $N_1$, $N_2$, $N_3$, $A\in 2^{\N_0}$, $A\ge 64$, 
and $16 \le |j_1-j_2|\le 32$.  
Assume that 
$\xi_1=(|\xi_1|\cos \theta_1, |\xi_1|\sin \theta_1)$ 
and $\xi_2=(|\xi_2|\cos \theta_2, |\xi_2|\sin \theta_2)$ 
satisfy $|\xi_1|\sim N_1$, $|\xi_2|\sim N_2$, $|\xi_1+\xi_2|\sim N_3$, 
$\theta_1\in \Theta_{j_1}^A$, and $\theta_2\in \Theta_{j_2}^A$. 
If $N_3\lesssim N_1\sim N_2$, then $A\gtrsim \frac{N_1}{N_3}$.
\end{lemm}
\begin{proof}
If $N_3\sim N_1$, then the claim is clearly true,
since $\frac{N_1}{N_3}\sim 1$.
So, we assume $N_3\ll N_1$. 
By using the rotation, 
we can assume $\theta_1\in \Theta_{j_1-j_2}^A$ 
and $\theta_2\in \Theta_{0}^A$. 
Then, it follows from $16\le |j_1-j_2|\le 32$, and \eqref{AP1a} that
\[
|\sin \theta_1|=|\sin (\pi +\theta_1)|
\ge \sin \left(\frac{|j_1-j_2|-2}{A}\pi\right)
\ge \sin \left(\frac{14}{A}\pi\right)
\ge \frac{7}{A} \pi
\]
and
\[
|\sin \theta_2|=|\sin (\pi +\theta_2)|
\le \sin \frac{2}{A}\pi
\le \frac{2}{A} \pi.
\]
We therefore obtain
\begin{align*}
|\xi_1+\xi_2|
&\ge ||\xi_1|\sin \theta_1+|\xi_2|\sin \theta_2| \\
&\ge |\xi_1| (|\sin \theta_1| - |\sin \theta_2|)
- |\xi_1+\xi_2| |\sin \theta_2| \\
&\ge |\xi_1|\left( \frac{7}{A}\pi-\frac{2}{A}\pi\right) - |\xi_1+\xi_2| \frac{2}{A} \pi,
\end{align*}
which yields that
\[
N_3 \sim |\xi_1+\xi_2| \gtrsim \frac{|\xi_1|}{A} \sim \frac{N_1}{A}.
\]
This shows the desired bound.
\end{proof}

We then present the proof of Proposition \ref{thm2.6}.

\begin{proof}[Proof of Proposition \ref{thm2.6}]
By Lemma~\ref{modul_est_2} and $L_{\max}\ll N_{\max}^2$, 
we can assume $N_3\lesssim N_1\sim N_2$. 
We divide the proof into the following two cases:
\begin{equation*}
\textnormal{(I)} \quad  L_{\textnormal{max}} \geq A^{-2} N_1^3N_3^{-1}, \qquad 
\textnormal{(\II)} \quad L_{\textnormal{max}} \leq  A^{-2} N_1^3N_3^{-1}.
\end{equation*}
We first consider the case $\textnormal{(I)}$. 
For the case $L_{\max}=L_3$, 
by using the (\ref{ABSE1}), we have
\begin{align*}
& \|Q_{L_3}^{-\sigma} P_{N_3}(R_{j_1}^A Q_{L_1}^{1}P_{N_1}u_{1}\cdot 
R_{j_2}^A Q_{L_2}^{-1}P_{N_2}u_{2})\|_{L^{2}_{tx}}  \\
\lesssim & \left(\frac{N_1}{N_3A}\right)^{\frac{1}{2}}  L_1^{\frac{1}{2}}L_2^{\frac{1}{2}}  
\|R_{j_1}^A Q_{L_1}^{1}P_{N_1}u_{1}\|_{L^2_{tx}}   \|R_{j_2}^A Q_{L_2}^{-1}P_{N_2}u_{2}\|_{L^2_{tx}}\\
\lesssim &  A^{\frac{1}{2}}N_1^{-1} L_1^{\frac{1}{2}}L_2^{\frac{1}{2}} L_3^{\frac{1}{2}}  
\|R_{j_1}^A Q_{L_1}^{\sigma_1}P_{N_1}u_{1}\|_{L^2_{tx}}   \|R_{j_2}^A Q_{L_2}^{\sigma_2}P_{N_2}u_{2}\|_{L^2_{tx}},
\end{align*}
which shows \eqref{0609}.
For $L_{\max}=L_1$, by the duality argument, H\"{o}lder inequality, and (\ref{ABSE2}), we have
\begin{align*}
& \|Q_{L_3}^{\sigma} P_{N_3}(R_{j_1}^A Q_{L_1}^{1}P_{N_1}u_{1}\cdot 
R_{j_2}^A Q_{L_2}^{-1}P_{N_2}u_{2})\|_{L^{2}_{tx}}\\
\sim & \sup_{\|u_3 \|_{L^2} =1} \left| \int_{\R \times \R^2} (R_{j_1}^A Q_{L_1}^{1}P_{N_1}u_{1}) \,  
(R_{j_2}^A Q_{L_2}^{-1}P_{N_2}u_{2}) \, (Q_{L_3}^{-\sigma} P_{N_3} u_3 ) \, dxdt \right|\\
\lesssim & \| R_{j_1}^A Q_{L_1}^{1}P_{N_1}u_{1} \|_{L^2} \\
&\qquad \times
\sup_{\|u_3 \|_{L^2} =1} \|R_{j_1}^A Q_{L_1}^{-1} P_{N_1}(R_{j_2}^A Q_{L_2}^{-1}P_{N_2}u_{2}\cdot 
 Q_{L_3}^{-\sigma}P_{N_3}u_{3})\|_{L^{2}_{tx}}  \\
\lesssim & A^{\frac{1}{2}}
N^{-1}_1 L_1^{\frac{1}{2}} L_2^{\frac{1}{2}}L_3^{\frac{1}{2}} 
\| R_{j_1}^A Q_{L_1}^{1}P_{N_1}u_{1} \|_{L^2_{tx}} 
\|R_{j_2}^A Q_{L_2}^{-1}P_{N_2}u_{2}\|_{L^2_{tx}}.
\end{align*}
The case $L_{\max}=L_2$ can be treated similarly. 

For $\textnormal{(\II)}$, by Plancherel's theorem and the duality argument, \eqref{0609} is verified by the following estimate:
\begin{equation}
\begin{split}
& \left| \int_{\R \times \R \times \R^2 \times \R^2}  f_1 (\ta_1, \xi_1) f_2 (\ta_2, \xi_2) f_3 (\tau_1+\tau_2, \xi_1+ \xi_2) d\ta_1 d\ta_2 d\xi_1 d \xi_2 \right| \\
&\lesssim 
A^{\frac{1}{2}} N_1^{-1} ( L_1 L_2 L_3)^\frac{1}{2}  
\|f_1\|_{L^2_{\tau \xi}} \|f_2 \|_{L^2_{\tau \xi}}  \|f_3 \|_{L^2_{\tau \xi}}
\notag
\end{split}
\end{equation}
where $f_{1}=\F_{tx} [R_{j_1}^A Q_{L_1}^{1}P_{N_1}u_{1}]$, 
$f_{2}=\F_{tx} [R_{j_2}^A Q_{L_2}^{-1}P_{N_2}u_{2}]$,  
and $f_{3}=\F_{tx} [Q_{L_3}^{-\sigma}P_{N_3}u_{3}]$.

Let $(\tau_i,\xi_i) \in \supp f_i$ for $i=1,2$ and $(\tau_1+\tau_2, \xi_1+\xi_2) \in \supp f_3$.
We write $\xi_i$ as
\[
\xi_i = (|\xi_i| \cos \theta_i, |\xi_i| \sin \theta_i).
\]
The assumption $16 \le |j_1-j_2| \le 32$ yields that $|\angle (\xi_1,\xi_2)|$ is confined to a set of measure $\sim A^{-1}$.
If $A \gg 1$,
from Lemma \ref{element1} with \eqref{AP1} and \eqref{AP2},
the range of $|\xi_2|$ is restricted to a set of measure
$\sim \frac{L_{\max}}{N_1}$
for any $(\ta_1, \xi_1) \in \supp f_1$.
Here,
the assumption (\II) and Lemma \ref{angle_est_A} yield that
\[
\frac{L_{\max}}{N_1}
\le A^{-2} N_1^2 N_3^{-1}
\lesssim A^{-1} N_1=: \delta.
\]
Namely, for any fixed $(\tau_1,\xi_1) \in \supp f_1$,
$|\xi_2|$ is restricted to a set of measure
$\sim \delta$.
If $A \sim 1$, we set $\delta := N_1$.
From $(\tau_2, \xi_2) \in \supp f_2$ 
and $N_1\sim N_2$,
$|\xi_2|$ is trivially restricted to a set of measure $\sim \delta$.

Next,
we decompose $f_1$ by thickened circular 
localization characteristic functions.
Namely,
by setting
\[
\mathbb{S}_\delta^R = \{ (\ta, \xi ) \in \R \cross \R^2 \ | \ R \leq \LR{\xi}\leq R + \delta \}
\]
for $R>0$,
we have
\begin{equation*}
f_1 = \sum_{k= 0}^{\left[\frac{N_1}{\delta}\right]+1} 
{\mathbf 1}_{\mathbb{S}_{\delta}^{N_1 + k \delta} } f_1,
\end{equation*}
where $[s]$ denotes the maximal integer which is not greater than $s \in \R$.
For $k=0,1, \dots, \left[\frac{N_1}{\delta}\right]+1$,
we may assume that $\supp f_2$ is confined to $\mathbb{S}_\delta^{\xi^0(k)}$ with some 
fixed $\xi^0(k) \sim N_2$.

Let
\begin{equation*}
C_{\delta}(\xi') := \{ (\ta, \xi) \in \R^3 \ | \ | \xi -\xi' | \leq \delta \} .
\end{equation*}
We apply a harmless decomposition to $f_1$, $f_2$, $f_3$ and assume that 
there exist $\xi^0_{f_i} \in \R^2$ such that 
$\supp f_i \subset C_{\delta}(\xi_{f_i}^0)$
for $i=1,2,3$.

We apply the same strategy as 
that of the proof of 
Proposition 4.4 in \cite{BHHT09}. 
Applying the transformation $\ta_1 = -|\xi_1|^2 + c_1$ and $\ta_2 = 
|\xi_2|^2 + c_2$ and Fubini's theorem, we find that it suffices to prove
\begin{equation}
\begin{split}
& \left| \int_{\R^2 \times \R^2}  f_1 (\phi_{c_1}^{-} (\xi_1)) f_2 (\phi_{c_2}^{+} (\xi_2)) 
f_3 (\phi_{c_1}^{-} (\xi_1) + \phi_{c_2}^{+} (\xi_2))  d\xi_1 d \xi_2 \right| \\
&\lesssim 
A^{\frac{1}{2}} N_1^{-1} 
\|f_1 \circ \phi_{c_1}^{-} \|_{L_\xi^2} \|f_2 \circ \phi_{c_2}^{+} \|_{L_\xi^2}  \|f_3 \|_{L^2_{\tau \xi}},
\label{2017-06-10a}
\end{split}
\end{equation}
where $f_3(\ta, \xi)$ is supported in $c_0 \leq \ta - \sigma |\xi|^2 \leq c_0 +1$ and 
\begin{equation}
\phi_{c_1}^{-} (\xi_1) = (-|\xi_1|^2 + c_1, \xi_1),\ \ 
\phi_{c_2}^{+} (\xi_2) = (|\xi_2|^2 + c_2, \xi_2).
\label{phi2}
\end{equation}
We use the scaling $(\ta, \, \xi) \to (N_1^2 \ta , \, N_1 \xi)$ to define
\begin{align*}
g_1 (\ta_1, \xi_1) &= f_1 (N_1^2 \ta_1, N_1 \xi_1), \\
g_2 (\ta_2, \xi_2) &= f_2 (N_1^2 \ta_2, N_1 \xi_2), \\ 
g_3 (\ta, \xi) &= f_3(N_1^2 \ta, N_1 \xi).
\end{align*}
If we set ${\widetilde{c_k}} = N_1^{-2} c_k$, inequality \eqref{2017-06-10a} reduces to
\begin{equation}
\begin{split}
 &\left| \int_{\R^2 \times \R^2}  g_1 (\phi_{\widetilde{c}_1}^{-} (\xi_1)) g_2 (\phi_{\widetilde{c}_2}^{+} (\xi_2)) 
g_3 (\phi_{\widetilde{c}_1}^{-} (\xi_1) + \phi_{\widetilde{c}_2}^{+} (\xi_2))  d\xi_1 d \xi_2 \right| \\
& \quad \lesssim 
A^{\frac{1}{2}} N_1^{-1} 
\| g_1 \circ \phi_{\widetilde{c}_1}^{-} \|_{L_\xi^2} \| g_2 \circ \phi_{\widetilde{c}_2}^{+} \|_{L_\xi^2}  \| g_3 \|_{L^2_{\tau \xi}}.
\notag
\end{split}
\end{equation}
Note that $g_3$ is supported in $S_3(N_1^{-2})$, where 
\begin{equation*}
S_3 (N_1^{-2}) = \left\{ (\ta, \xi) \in C_{N_1^{-1} \delta}(N_1^{-1} \xi_{f_3}^0) 
\ | \ \sigma|\xi|^2 +\frac{c_0}{N_1^2} \leq \ta  \leq \sigma|\xi|^2 +\frac{c_0+1}{N_1^2} \right\}.
\end{equation*}
By density and duality arguments,
it suffices to show for continuous 
$g_1$ and $g_2$ that
\begin{equation}
\| g_1 |_{S_1} * g_2 |_{S_2} \|_{L^2(S_3 (N_1^{-2}))} 
\lesssim A^{\frac{1}{2}} N_1^{-1} 
\| g_1 \|_{L^2(S_1)} \| g_2 \|_{L^2(S_2)},
\label{2017-06-10c}
\end{equation}
where $S_1$ and $S_2$ denote the following surfaces 
\begin{align*}
S_1 &= \{ \phi_{\widetilde{c}_1}^{-} (\xi_1) | \ \xi _1 \in \R^2 \} \cap 
C_{N_1^{-1} \delta}(N_1^{-1} \xi_{f_1}^0), \\
S_2 &=
\{ \phi_{\widetilde{c}_2}^{+} (\xi_2) | \ \xi _2 \in \R^2 \} \cap
C_{N_1^{-1} \delta}(N_1^{-1} \xi_{f_2}^0).
\end{align*}
Then, \eqref{2017-06-10c} is immediately obtained by
\begin{equation}
\| g_1 |_{S_1} * g_2 |_{S_2} \|_{L^2(S_3)} 
\lesssim A^{\frac{1}{2}}  
\| g_1 \|_{L^2(S_1)} \| g_2 \|_{L^2(S_2)},
\label{2017-06-10d}
\end{equation}
where 
\begin{equation*}
S_3 = \Big\{ (\psi (\xi), \xi) \in C_{N_1^{-1} \delta}(N_1^{-1} \xi_{f_3}^0) 
\ | \ \psi(\xi) =  \sigma|\xi|^2 +\frac{c_0}{N_1^2} \Big\}.
\end{equation*}
Since $|N_1^{-1} \xi_{f_1}^0| \sim |N_1^{-1} \xi_{f_2}^0| \sim 1$,
$|N_1^{-1} \xi_{f_3}^0| \sim N_1^{-1}N_3\lesssim 1$,
and $N^{-1} \delta \sim A^{-1}$,
we have
\[
\textnormal{diam}(S_i) \lesssim A^{-1}
\]
for $i=1,2,3$.
By applying a harmless decomposition,
we can assume 
\begin{equation}
\textnormal{diam}(S_i) \leq 2^{-10} \langle \sigma \rangle^{-1} A^{-1}
\label{2017-06-10e}
\end{equation}
for $i=1,2,3$.

For any $\lambda_i \in S_i$, $i=1,2,3$, there exist $\xi_1$, $\xi_2$, $\xi$ such that
\begin{equation*}
\lambda_1=\phi_{{{\widetilde{c}_1}}}^{-} (\xi_1), \quad \lambda_2 =  \phi_{{{\widetilde{c}_2}}}^{+} (\xi_2), \quad \lambda_3 = (\psi (\xi), \xi).
\end{equation*}
By \eqref{phi2}, the unit normals ${\mathfrak{n}}_i$ on $\lambda_i$ are written as
\begin{align*}
& {\mathfrak{n}}_1(\lambda_1) = \frac{1}{\sqrt{1+4|\xi_1|^2}} 
\left(1, \ 2\xi_1^{(1)}, \ 2\xi_1^{(2)} \right), \\
& {\mathfrak{n}}_2(\lambda_2) = \frac{1}{\sqrt{1+4|\xi_2|^2}} 
\left(1, \ -2\xi_2^{(1)}, \ -2\xi_2^{(2)} \right), \\
& {\mathfrak{n}}_3(\lambda_3) = \frac{1}{\sqrt{1+4\sigma^2|\xi|^2}} \ 
\left(-1, \ 2 {\sigma} \xi^{(1)}, \ 2 {\sigma} \xi^{(2)} \right),
\end{align*}
where $\xi^{(i)}$ $(i=1,2)$ denotes the $i$-th component of $\xi$. 
Clearly, the surfaces $S_1$, $S_2$, $S_3$ satisfy the following 
H\"{o}lder condition.
\begin{equation}
\sup_{\lambda_i, \lambda_i' \in S_i}
\bigg(
\frac{|\mathfrak{n}_i(\lambda_i) - 
\mathfrak{n}_i(\lambda_i')|}{|\lambda_i - \lambda_i'|}
+ \frac{|\mathfrak{n}_i(\lambda_i) \cdot (\lambda_i - \lambda_i')|}{|\lambda_i - \lambda_i'|^2}
\bigg)
\leq 2^3.\label{aiueo1}
\end{equation}
We may assume that there exist $\xi_1', \xi_2', \xi' \in \R^2$ such that
\begin{equation*}
\xi_1' + \xi_2' = \xi', \quad \phi_{{{\widetilde{c}_1}}}^{-} (\xi_1') \in S_1, \ \phi_{{{\widetilde{c}_2}}}^{+} (\xi_2') \in S_2, \ (\psi (\xi'), \xi') \in S_3,
\end{equation*}
otherwise the left-hand side of \eqref{2017-06-10d} vanishes. 
Let $\lambda_1' = \phi_{{{\widetilde{c}_1}}}^{-}(\xi_1')$, $\lambda_2' = \phi_{{{\widetilde{c}_2}}}^{+} (\xi_2')$, $\lambda_3' =  (\psi (\xi'), \xi') $. 
For any $\lambda_1 = \phi_{{{\widetilde{c}_1}}}^{-} (\xi_1) \in S_1$, 
we deduce from $\lambda_1$, $\lambda_1' \in S_1$, \eqref{2017-06-10e}, and \eqref{aiueo1} that
\begin{equation}
|{\mathfrak{n}}_1(\lambda_1) - {\mathfrak{n}}_1(\lambda_1')| \leq 2^{-7} 
\langle \sigma \rangle^{-1} A^{-1}.\label{2017-06-10f}
\end{equation}
Similarly, for any $\lambda_2 \in S_2$ and $\lambda_3 \in S_3$, we have
\begin{align}
& |{\mathfrak{n}}_2(\lambda_2) - {\mathfrak{n}}_2(\lambda_2')| \leq 2^{-7} 
\langle \sigma \rangle^{-1}  A^{-1},
\label{2017-06-10g}\\
& |{\mathfrak{n}}_3(\lambda_3) - {\mathfrak{n}}_3(\lambda_3')| \leq 2^{-7} 
\langle \sigma \rangle^{-1} A^{-1}.\label{2017-06-10h}
\end{align}
From \eqref{2017-06-10e} and \eqref{aiueo1}, 
once the following transversality condition
\begin{equation} 
 \frac{\langle \sigma \rangle^{-1}  A^{-1}}{32} \leq |\textnormal{det} N(\lambda_1, \lambda_2, \lambda_3)| \quad 
\textnormal{for any} \ \lambda_i \in S_i
\label{trans}
\end{equation}
is verified, we obtain the desired estimate \eqref{2017-06-10d}  by applying Proposition \ref{prop2.7} with $a = \langle \sigma \rangle^{-1} A^{-1}/32$.%
\footnote{Strictly speaking, we need to construct a larger set $S_i^\ast$ and replace $S_i$ by $S_i^\ast$.
However, since $S_i$ is a graph in this setting,
this is a sight modification.
Indeed, by setting $U_1 = C_{2N_1^{-1} \delta}(N_1^{-1} \xi_{f_1}^0)$ and $S_1^\ast = \{ \phi_{\widetilde{c}_1}^{-} (\xi_1) | \ \xi _1 \in \R^2 \} \cap U_1$, we have $\dist (S_1, U_1^c) \ge 2 N^{-1} \delta = 2 A^{-1}$.
Similarly, we can set $S_2^\ast$ and $S_3^\ast$.
Moreover, since estimates with $S_i$ replaced by $S_i^\ast$ are similarly obtained, we omit the details.
}

Finally, we show \eqref{trans}.
For $\lambda_1' = \phi_{\widetilde{c}_1}^{-} (\xi_1')$, $\lambda_2' = \phi_{\widetilde{c}_2}^{+} (\xi_2')$, $\lambda_3' =  (\psi(\xi'), \xi') $ and 
$\xi_1'+\xi_2' = \xi'$,
a direct calculation shows that
\begin{align*}
|\textnormal{det} N(\lambda_1', \lambda_2', \lambda_3')| \geq & 
\frac{\langle \sigma \rangle^{-1}}{\LR{2 |\xi'|}} 
\frac{1}{\LR{2|\xi_1'|}} \frac{1}{\LR{2|\xi_2'|}}  \left|\textnormal{det}
\begin{pmatrix}
1 & 1 & - 1 \\
 \xi_1'^{(1)}  &   -\xi_2'^{(1)}  & {\sigma} \xi'^{(1)} \\
 \xi_1'^{(2)}  & -\xi_2'^{(2)}   & {\sigma} \xi'^{(2)} 
\end{pmatrix} \right| \notag \\
\geq & \frac{1}{8}\langle \sigma \rangle^{-1}
 \left| \frac{\xi_1'^{(1)}  \xi_2'^{(2)}  - \xi_1'^{(2)} \xi_2'^{(1)}}{|\xi_1'||\xi_2'|} \right|  \notag\\
\geq &  
\frac{\langle \sigma \rangle^{-1}A^{-1}}{16}.
\end{align*}
Combining this bound with \eqref{2017-06-10f}--\eqref{2017-06-10h},
we obtain \eqref{trans}.
This concludes the proof of Proposition \ref{thm2.6}.
\end{proof}

%

%
%
%
%
\section{Proof of bilinear estimates for $d=3$\label{be_for_3d}}
In this section, 
we prove Proposition~\ref{key_be} for $d=3$. 
We first give the operators with respect to angular variables  introduced in \cite{BH11}. 
\begin{defn}[\cite{BH11}]
For each $A \in \N$, $\{\omega_A^j \}_{j \in \Omega_A} $ denotes a set of spherical caps of ${\BBB S}^2$ with the following properties:\\
(i) The angle $\angle{(x,y)}$ between 
any two vectors in $x$, $y \in \omega_A^j$ satisfies
\begin{equation*}
\left| \angle{(x,y)} \right| \leq A^{-1}.
\end{equation*}
(ii) Characteristic functions $\{ {\mathbf 1}_{\omega_A^j} \}$ satisfy
\begin{equation*}
1 \leq \sum_{j \in \Omega_A} {\mathbf 1}_{\omega_A^j}(x) \leq 3
\end{equation*}
for any $x \in {\BBB S}^2$.

We define the function
\begin{equation*}
\alpha (j_1,j_2) = \inf \left\{ \left| \angle{( \pm x, y)} \right| : \ x \in \omega_A^{j_1}, \ y \in \omega_A^{j_2} \right\},
\end{equation*}
which measures the minimal angle between any two straight lines through the spherical caps $\omega_A^{j_1}$ and 
$\omega_A^{j_2}$, respectively. It is easily observed that for any fixed $j_1 \in \Omega_A$ there exist only a finite number of $j_2 \in \Omega_A$ which satisfies $\alpha (j_1,j_2) \sim A^{-1}$.

Based on the above construction, for each $j \in \Omega_A$,
we define 
\begin{equation*}
{\mathfrak{D}}_{j}^A = \left\{ (\ta, \xi) \in \R \cross (\R^3 \setminus \{0\}) \, : \, 
 \frac{\xi}{|\xi|} \in \omega_A^j \right\}
\end{equation*}
and the corresponding localization operator
\begin{equation*}
\F (R_j^A u) (\tau, \xi) = \frac{\chi_{\omega_j^A} (\frac{\xi}{|\xi|})}{\chi  (\frac{\xi}{|\xi|})} \F u (\tau , \xi).
\end{equation*}
\end{defn}
The following estimates are three dimensional version of 
Propositions~\ref{BSE_HHL}, ~\ref{BSE_HHL_2}, 
and ~\ref{thm2.6}. 
\begin{prop}\label{BSE_HHL_3d}
Let $N_1$, $N_2$, $N_3$, $L_1$, $L_2$, $L_3$, $A\in 2^{\N_0}$, and
$j_1$, $j_2\in \Omega_A$.
We assume $A\ge 64$, $\alpha (j_1,j_2)\lesssim A^{-1}$, and $N_3\lesssim N_1\sim N_2$. 
Then,  
we have the following estimate:
\begin{align}
\begin{split}
&\|P_{N_3}(R_{j_1}^AQ_{L_1}^{1}P_{N_1}u_1\cdot R_{j_2}^AQ_{L_2}^{-1}P_{N_2}u_2)\|_{L^2_{tx}(\R\times \R^3)}\\
&\hspace{3ex}
\lesssim \frac{N_1}{N_3^{\frac{1}{2}}A}
L_1^{\frac{1}{2}}L_2^{\frac{1}{2}}
\|R_{j_1}^AQ_{L_1}^{1}P_{N_1}u_1\|_{L^2_{tx}(\R\times \R^3)}
\|R_{j_2}^AQ_{L_2}^{-1}P_{N_2}u_2\|_{L^2_{tx}(\R\times \R^3)}.
\end{split} \label{ABSE1_3d}
\end{align}
\end{prop}
\begin{prop}\label{BSE_HHL_2_3d}
Let $\sigma \in \R\backslash \{0,\pm 1\}$.
Let $N_1$, $N_2$, $N_3$, $L_1$, $L_2$, $L_3$, $A\in 2^{\N_0}$, and
$j_1$, $j_2\in \Omega_A$.
We assume $L_{\textnormal{max}} \ll N_{\max}^2$, $A\ge 64$, and $\alpha (j_1,j_2)\lesssim A^{-1}$. 
Then,  
we have the following estimate:
\begin{align}
\begin{split}
&\|R_{j_1}^AQ_{L_1}^{-1}P_{N_1}(R_{j_2}^AQ_{L_2}^{-1}P_{N_2}u_2\cdot Q_{L_3}^{-\sigma}P_{N_3}u_3)\|_{L^2_{tx}(\R\times \R^3)}\\
&\hspace{3ex}
\lesssim A^{-1}N_{\max}^{\frac{1}{2}}
L_2^{\frac{1}{2}}L_3^{\frac{1}{2}}
\|R_{j_2}^AQ_{L_2}^{-1}P_{N_2}u_2\|_{L^2_{tx}(\R\times \R^3)}
\|Q_{L_3}^{-\sigma}P_{N_3}u_3\|_{L^2_{tx}(\R\times \R^3)}.
\end{split}\label{ABSE2_3d}
\end{align}
\end{prop}
\begin{prop}\label{thm2.6_3d}
Let $\sigma \in \R\backslash \{0,\pm 1\}$.
Let $N_1$, $N_2$, $N_3$, $L_1$, $L_2$, $L_3$, $A\in 2^{\N_0}$, 
and $j_1$, $j_2\in \Omega_A$. 
We assume $L_{\textnormal{max}} \ll N_{\max}^2$, 
$A\ge 64$, and 
$\alpha (j_1,j_2)\sim A^{-1}$. 
Then the following estimate holds:
\begin{equation}
\begin{split}
&\|Q_{L_3}^{\sigma} P_{N_3}(R_{j_1}^A Q_{L_1}^{1}P_{N_1}u_{1}\cdot 
R_{j_2}^A Q_{L_2}^{-1}P_{N_2}u_{2})\|_{L^{2}_{tx}(\R\times \R^3)} \\
& \quad
\lesssim N_{\max}^{-\frac{1}{2}}L_1^{\frac{1}{2}}L_2^{\frac{1}{2}} L_3^{\frac{1}{2}} 
\|R_{j_1}^A Q_{L_1}^{1}P_{N_1}u_{1}\|_{L^2_{tx}(\R\times \R^3)} 
\|R_{j_2}^A Q_{L_2}^{-1}P_{N_2}u_{2}\|_{L^2_{tx}(\R\times \R^3)} .
\end{split}
\notag
\end{equation}
\end{prop}
If $\supp f_i \subset {\mathfrak{D}}_{j_i}^A$ $(i=1,2)$, after applying rotation in space and suitable decomposition, we may assume that the supports of $f_1$ and $f_2$ are both contained in the following slab
\begin{equation*}
\Sigma_3 (N_1 A^{-1}) := \{ (\ta, \widetilde{\xi}, \xi^{(3)}) \in \R \cross \R^2 \cross \R \ | \ |\xi^{(3)}| \leq  N_1 A^{-1} \}.
\end{equation*}
Set
\[
\psi_{N_3, L_3}^{\sigma}(\ta,\xi)=\psi_{L_3} (\tau - \sigma |\xi|^2) \psi_{N_3}(\xi),
\] 
where $\psi$ is as in the notation at the end of Section \ref{intro}.
We claim that if
\begin{equation}
\begin{split}
& \bigg\| \psi_{N_3, L_3}^{\sigma}(\ta,\widetilde{\xi},\xi^{(3)})
\int_{\R \times \R^2}  
f_1 (\ta_1, \widetilde{\xi}_1, \xi_1^{(3)}) \, 
f_2 (\ta - \ta_1, \widetilde{\xi} - \widetilde{\xi}_1, \xi^{(3)} - \xi_1^{(3)}) d\ta_1 d \widetilde{\xi}_1 \bigg\|_{L_{\tau \widetilde{\xi}}^2(\R\times \R^2)}\\
& \qquad \qquad \qquad \qquad \qquad  \lesssim 
K\|f_1 ( \xi_1^{(3)}) \|_{L_{\tau \widetilde{\xi}}^2} 
\|f_2 (\xi^{(3)} - \xi_1^{(3)}) \|_{L_{\tau \widetilde{\xi}}^2} 
\end{split}\label{3d_est_ess}
\end{equation}
holds uniformly for $\xi^{(3)}$ and $\xi_1^{(3)}$ with $|\xi^{(3)} - \xi_1^{(3)}| \leq  N_1 A^{-1}$ and $|\xi_1^{(3)}| \leq  N_1 A^{-1}$, 
then we can obtain
\begin{equation}
\begin{split}
 \left\| \psi_{N_3, L_3}^{\sigma}(\ta,\xi) \int_{\R \times \R^3} 
f_1 (\tau_1,\xi_1) f_2 (\tau-\tau_1, \xi-\xi_1)d\tau_1d\xi_1 \right\|_{L_{\tau \xi}^2(\R\times \R^3)} & \\ 
\lesssim 
A^{-\frac{1}{2}} N_{1}^{\frac{1}{2}}K
\|f_1\|_{L^2_{\tau \xi}} & \|f_2 \|_{L^2_{\tau \xi}}.
\end{split}
\notag
\end{equation}
Indeed, once  \eqref{3d_est_ess} holds, 
from Minkowski's inequality and Young's inequality, we have
\begin{align*}
&  \left\|  \psi_{N_3, L_3}^{\sigma}(\ta,\xi) \int_{\R \times \R^3}  
f_1 (\tau_1,\xi_1) f_2 (\tau-\tau_1, \xi-\xi_1)d\tau_1d\xi_1 \right\|_{L_{\tau \xi}^2} \\
= \ & \bigg\|  \psi_{N_3, L_3}^{\sigma}(\ta,\widetilde{\xi},\xi^{(3)}) \\
& \qquad \times
\int_{\R \times \R^3}  f_1 (\ta_1, \widetilde{\xi}_1, \xi_1^{(3)})  \, 
f_2 (\ta - \ta_1, \widetilde{\xi} - \widetilde{\xi}_1, \xi^{(3)} - \xi_1^{(3)}) d\ta_1 d \xi_1 \bigg\|_{L_{\tau \xi}^2}\\
\lesssim \ & \bigg\| \int_\R \Big\|  \psi_{N_3, L_3}^{\sigma}(\ta,\widetilde{\xi},\xi^{(3)})
\int_{\R \times \R^2} 
f_1 (\ta_1, \widetilde{\xi}_1, \xi_1^{(3)})   \\
& \times \, 
f_2 (\ta - \ta_1, \widetilde{\xi} - \widetilde{\xi}_1, \xi^{(3)} - \xi_1^{(3)}) d\ta_1 d \widetilde{\xi}_1 \Big\|_{L_{\tau \widetilde{\xi}}^2}  
d\xi_1^{(3)} 
\bigg\|_{L_{\xi^{(3)} }^2}\\
 \underset{\eqref{3d_est_ess}}{\lesssim} & K\left\| \int_\R \|f_1  ( \xi_1^{(3)}) \|_{L_{\tau \widetilde{\xi}}^2} 
\|f_2 (\xi^{(3)} - \xi_1^{(3)}) \|_{L_{\tau \widetilde{\xi}}^2} d\xi_1^{(3)} 
\right\|_{L_{\xi^{(3)} }^2}\\
\lesssim \ &  
K\bigg(
\sup_{\xi^{(3)}} \int_\R \|f_1  ( \xi_1^{(3)} ) \|_{L_{\tau \widetilde{\xi}}^2} 
\|f_2 (\xi^{(3)} - \xi_1^{(3)}) \|_{L_{\tau \widetilde{\xi}}^2} d\xi_1^{(3)}
\bigg) \\
&\quad \times
\|\ee_{\{|\xi^{(3)}|\lesssim N_1A^{-1}\}}\|_{L^2_{\xi^{(3)}}}\\
\lesssim \ &  A^{-\frac{1}{2}} N_1^{\frac{1}{2}} 
K\|f_1 \|_{ L_{\tau \xi}^2} 
\|f_2 \|_{ L_{\ta \xi}^2}. 
\end{align*}
Therefore, to show 
Propositions~\ref{BSE_HHL_3d}, ~\ref{BSE_HHL_2_3d}, 
~\ref{thm2.6_3d}, 
it suffices to prove (\ref{3d_est_ess}) 
for
\[
K=\left(\frac{N_1}{N_3A}\right)^{\frac{1}{2}}L_1^{\frac{1}{2}}L_2^{\frac{1}{2}},
\quad
A^{-\frac{1}{2}} L_2^{\frac{1}{2}}L_3^{\frac{1}{2}},
\quad \text{ and } \quad
A^{\frac{1}{2}}N_{\max}^{-1}L_1^{\frac{1}{2}}L_2^{\frac{1}{2}}L_3^{\frac{1}{2}},
\]
respectively. 
Since we can get these estimates by 
similar argument as the proof of 
Proposition~\ref{BSE_HHL}, ~\ref{BSE_HHL_2}, 
and ~\ref{thm2.6} (See, also \cite{HK}), 
we omit its proof. 

Now, we give the proof of the bilinear estimates. 
\begin{proof}[Proof of Proposition~\ref{key_be} for $d=3$]
Let $1>s>\frac{1}{2}$ and
\[
(s_1,s_2,s_3)\in \{(s,s,-s),(s,-s,s),(-s,s,s)\}.
\]
We prove (\ref{key_est}). 
We set
\[
u_{N_1,L_1}:=Q_{L_1}^{1}P_{N_1}U,\ 
v_{N_2,L_2}:=Q_{L_2}^{-1}P_{N_2}V,\ 
w_{N_3,L_3}:=Q_{L_3}^{-\sigma}P_{N_3}W. 
\]
Then, we have
\[
\begin{split}
&
\left|\int_{\R \times \R^3} U(t,x)V(t,x)\p_jW(t,x)dxdt\right|\\
&\lesssim \sum_{N_1,N_2,N_3\ge 1}\sum_{L_1,L_2, L_3 \ge 1} N_{3}
\left|\int_{\R \times \R^3} u_{N_1,L_1}v_{N_2,L_2}w_{N_3,L_3}dxdt\right|. 
\end{split}
\]
By the same reason for $d=2$, it suffices to show that
\begin{equation}\label{desired_est_2}
\begin{split}
&N_{3}
\left|\int_{\R \times \R^3} u_{N_1,L_1}v_{N_2,L_2}w_{N_3,L_3}dxdt\right|\\
&\lesssim
N_{\min}^{s}(L_1L_2L_3)^{c}\left(\frac{N_{\min}}{N_{\max}}\right)^{\eps}
\|u_{N_1,L_1}\|_{L^2_{tx}}\|v_{N_2,L_2}\|_{L^2_{tx}}\|w_{N_3,L_3}\|_{L^2_{tx}}
\end{split}
\end{equation}
for some $b'\in (0,\frac{1}{2})$, $c\in (0,b')$, and $\eps >0$. 

Now, we prove (\ref{desired_est_2}).\\
\kuuhaku \\
\underline{Case\ 1:\ High modulation, $\displaystyle L_{\max}\gtrsim N_{\max}^2$}

We first assume $L_3=L_{\max}$. 
By the symmetry, we can assume $N_1\le N_2$.
Then, by the Cauchy-Schwarz inequality and (\ref{L2be_est_2}), 
we have
\[
\begin{split}
&N_3\left|\int_{\R \times \R^3} u_{N_1,L_1}v_{N_2,L_2}w_{N_3,L_3}dxdt\right|\\
&\lesssim N_{3}\|P_{N_3}(u_{N_1,L_1}v_{N_2,L_2})\|_{L^2_{tx}}\|w_{N_3,L_3}\|_{L^2_{tx}}\\
&\lesssim N_{3}N_{1}^{\frac{1}{2}+4\delta}
\left(\frac{N_{1}}{N_{2}}\right)^{\frac{1}{2}- 2\delta}L_2^{c}L_3^{c}
\|u_{N_1,L_1}\|_{L^2_{tx}}\|v_{N_2,L_2}\|_{L^2_{tx}}\|w_{N_3,L_3}\|_{L^2_{tx}}\\
&\lesssim N_{3}N_{1}^{\frac{1}{2}+4\delta}
\left(\frac{N_{1}}{N_{2}}\right)^{\frac{1}{2}- 2\delta}
N_{\max}^{-2c}(L_1L_2L_3)^{c}\\
&\hspace{15ex}\times \|u_{N_1,L_1}\|_{L^2_{tx}}\|v_{N_2,L_2}\|_{L^2_{tx}}\|w_{N_3,L_3}\|_{L^2_{tx}}
,
\end{split}
\]
where $\delta = \frac{1}{2}-c$. 
If $N_3\lesssim N_1\sim N_2$, then we obtain
\[
N_{3}N_{1}^{\frac{1}{2}+4\delta}
\left(\frac{N_{1}}{N_{2}}\right)^{\frac{1}{2}- 2\delta}
N_{\max}^{-2c}
\sim N_3^{\frac{7}{2}-6c-s}N_3^s\left(\frac{N_3}{N_1}\right)^{6c-\frac{5}{2}}. 
\]
If $N_1\lesssim N_2\sim N_3$, then we obtain
\[
N_{3}N_{1}^{\frac{1}{2}+4\delta}
\left(\frac{N_{1}}{N_{2}}\right)^{\frac{1}{2}- 2\delta}
N_{\max}^{-2c}
\sim N_1^{\frac{7}{2}-6c-s}N_1^s\left(\frac{N_1}{N_3}\right)^{4c-\frac{3}{2}}.
\]
Therefore, by choosing $b'$ and $c$ as
$\max\{\frac{1}{6}(\frac{7}{2}-s),\frac{5}{12},\frac{3}{8}\}<c<b'<\frac{1}{2}$
for $s>\frac{1}{2}$, 
we get \eqref{desired_est_2}. 
The case $L_1=L_{\max}$ and $L_2=L_{\max}$ is similarly handled, 
but we use (\ref{L2be_est_20}) 
instead of  (\ref{L2be_est_2}). \\
\kuuhaku \\
\underline{Case\ 2:\ Low modulation, $\displaystyle L_{\max}\ll  N_{\max}^2$}

By Lemma \ref{modul_est_2}, 
we can assume $N_3 \lesssim N_1 \sim N_2$. 
We set
\begin{equation}
M=N_1^{\frac{3}{2}}N_3^{-s}
\label{mod3}
\end{equation}
and 
decompose $\R^4 \cross \R^4$ as follows:
\begin{equation*}
\R^4 \cross \R^4 =
\bigg(\bigcup_{\tiny{\substack{j_1,j_2\in \Omega_{M}\\ \alpha(j_1,j_2)\lesssim M^{-1}}}} 
{\mathfrak{D}}_{j_1}^{M} \cross {\mathfrak{D}}_{j_2}^{M} \bigg)
 \cup 
\bigg(\bigcup_{64 \leq A \leq M} \ \bigcup_{\tiny{\substack{j_1,j_2\in \Omega_A\\ \alpha(j_1,j_2)\sim A^{-1}}}} 
{\mathfrak{D}}_{j_1}^A \cross {\mathfrak{D}}_{j_2}^A\bigg).
\end{equation*}
We can write
\begin{align}
 &\left|\int_{\R \times \R^3} u_{N_1,L_1}v_{N_2,L_2}w_{N_3,L_3}dxdt\right| \notag\\
 &\leq \sum_{{\tiny{\substack{j_1,j_2 \in \Omega_{M} \notag\\ 
\alpha(j_1,j_2)\lesssim M^{-1}}}}} 
\left|\int_{\R \times \R^3} u_{N_1,L_1, j_1}^M v_{N_2,L_2, j_2}^M w_{N_3,L_3}dxdt\right| \notag\\ 
&\ \ \ \ \ \ \ \ \ \ \ \ + 
\sum_{64 \leq A \leq M}  \sum_{{\tiny{\substack{j_1,j_2\in \Omega_A\\ \alpha(j_1,j_2)\sim A^{-1}}}}} 
\left|\int_{\R \times \R^3} u_{N_1,L_1, j_1}^A v_{N_2,L_2, j_2}^A w_{N_3,L_3}dxdt\right| \notag\\
&=: {\rm I} + {\rm \II},
\label{I,II2}
\end{align}
where
\[
u_{N_1,L_1,j_1}^A=R^A_{j_1}u_{N_1,L_1},\ \ 
v_{N_2,L_2,j_2}^A=R^A_{j_2}v_{N_2,L_2}. 
\]
For the first term $\rm I$ in \eqref{I,II2}, 
we first assume $L_{\textnormal{max}} = L_3$. 
By \eqref{I,II2}, the H\"older inequality, (\ref{ABSE1_3d}), and \eqref{mod3},
we get
\begin{align*}
N_3 \cdot {\rm I}
& \lesssim  
\sum_{{\tiny{\substack{j_1,j_2 \in \Omega_{M}\\ 
\alpha(j_1,j_2)\lesssim M^{-1}}}}}
N_3\|P_{N_3} ( u_{N_1,L_1, j_1}^M v_{N_2,L_2, j_2}^M)\|_{L^{2}_{tx}}
\|w_{N_3,L_3} \|_{{L^2_{t x}}}  \\
& \lesssim 
\sum_{{\tiny{\substack{\\ j_1,j_2 \in \Omega_{M}\\ 
\alpha(j_1,j_2)\lesssim M^{-1}}}}}N_3
\frac{N_1}{N_3^{\frac{1}{2}}M} L_1^{\frac{1}{2}}L_2^{\frac{1}{2}} 
\|u_{N_1,L_1, j_1}^M\|_{L^2_{tx}}   \|v_{N_2,L_2, j_2}^M\|_{L^2_{tx}}
\|w_{N_3,L_3} \|_{{L^2_{t x}}}\\
& \lesssim  N_3^{s}\left(\frac{N_3}{N_1}\right)^{\frac{1}{2}}
  L_1^{\frac{1}{2}}L_2^{\frac{1}{2}}
 \|u_{N_1,L_1}\|_{L^2_{tx}}\|v_{N_2,L_2}\|_{L^2_{tx}}\|w_{N_3,L_3}\|_{L^2_{tx}}\\
&\lesssim N_3^{s}\left(\frac{N_3}{N_1}\right)^{\frac{1}{2}} (L_1L_2L_3)^{\frac{1}{3}} \|u_{N_1,L_1}\|_{L^2_{tx}}\|v_{N_2,L_2}\|_{L^2_{tx}}\|w_{N_3,L_3}\|_{L^2_{tx}},
\end{align*}
which shows \eqref{desired_est_2}.
If $L_{\max}=L_1$ or $L_{\max}=L_2$, 
then we use (\ref{ABSE2_3d}) instead of (\ref{ABSE1_3d}). 

For the second term $\rm \II$, by \eqref{I,II2}, Proposition \ref{thm2.6_3d}, $L_{1}L_2L_3\ll N_1^6$, and \eqref{mod3},
we get
\begin{align*}
N_3 \cdot {\rm \II}
& \lesssim  
\sum_{64 \leq A \leq M}  \sum_{{\tiny{\substack{j_1,j_2\in \Omega_A\\ \alpha(j_1,j_2)\sim A^{-1}}}}} 
N_3\|Q_{L_3}^{\sigma} P_{N_3} ( u_{N_1,L_1, j_1}^A v_{N_2,L_2, j_2}^A)\|_{L^{2}_{tx}} 
\|w_{N_3,L_3} \|_{{L^2_{t x}}}\\
& \lesssim 
\sum_{64 \leq A \leq M}  \sum_{{\tiny{\substack{j_1,j_2\in \Omega_A\\ \alpha(j_1,j_2)\sim A^{-1}}}}} 
 N_3N_1^{-\frac{1}{2}} ( L_1 L_2 L_3)^\frac{1}{2} \\
&\hspace*{100pt}
\times
\|u_{N_1,L_1, j_1}^A\|_{L^2_{tx}}   \|v_{N_2,L_2, j_2}^A\|_{L^2_{tx}}
\|w_{N_3,L_3} \|_{{L^2_{t x}}}\\
&  \lesssim N_3N_1^{-\frac{1}{2}}(\log M)
( L_1 L_2 L_3)^{\frac{1}{2}} \|u_{N_1,L_1}\|_{L^2_{tx}}\|v_{N_2,L_2}\|_{L^2_{tx}}\|w_{N_3,L_3}\|_{L^2_{tx}}\\
&\lesssim N_3^{s}\left(\frac{N_3}{N_1}\right)^{1-s}( L_1 L_2 L_3)^{\frac{1}{2}-\frac{1}{12}(s-\frac{1}{2})} \\
&\qquad \times 
\|u_{N_1,L_1}\|_{L^2_{tx}}\|v_{N_2,L_2}\|_{L^2_{tx}}\|w_{N_3,L_3}\|_{L^2_{tx}},
\end{align*} 
here we used $\log M\lesssim \log N_1\lesssim N_1^{\frac{1}{2}(s-\frac{1}{2})}$ for $s>\frac{1}{2}$. 
This shows \eqref{desired_est_2}, and we hence conclude the proof of Proposition \ref{key_be} for $d=3$.
\end{proof}
%
%
%
\section{The lack of the third differentiability of the flow map}
For $\sigma \in \R\backslash\{0\}$, 
we define the Duhamel integral operator $\It_{\sigma}$ as
\begin{equation}
\It_{\sigma}(f)(t)
:=\int_0^te^{i(t-t')\sigma \partial_x^2} f(t') dt'.
\label{iter0}
\end{equation}
We also define the iteration terms for \eqref{NLS_sys} as follows:
\begin{equation}
\begin{gathered}
u^{(1)}(t):=e^{it\alpha \partial_x^2}u_0,\quad
v^{(1)}(t):=e^{it\beta \partial_x^2}v_0, \quad
w^{(1)}(t):=e^{it\gamma \partial_x^2}w_0,
\\
u^{(k)}(t) := i \sum_{\substack{k_1,k_2 \in \N \\ k_1+k_2=k}}
\It_{\alpha}(\dx w^{(k_1)} \cdot v^{(k_2)})(t),\\
v^{(k)}(t) := i \sum_{\substack{k_1,k_2 \in \N \\ k_1+k_2=k}}
\It_{\beta}(\dx \overline{w^{(k_1)}} \cdot u^{(k_2)})(t),\\
w^{(k)}(t) := -i \sum_{\substack{k_1,k_2 \in \N \\ k_1+k_2=k}}
\It_{\gamma}( \dx (u^{(k_1)} \overline{v^{(k_2)}}))(t)
\end{gathered}
\label{iter}
\end{equation}
for any integer $k$ greater than $1$.

The following proposition implies Theorem~\ref{NotC3}.
\begin{prop}\label{notC3}
Let $d=1$, $0<T\ll 1$, and let $\mu$ be as in \eqref{coeff}.
Assume $\mu >0$ and $s<0$. 
For every $C>0$, there exist $u_0$, $v_0$, $w_0\in H^{s}(\R)$ such that
\begin{equation}\label{flow_map_estimate}
\sup_{0\le t\le T}\|u^{(3)}(t)\|_{H^s}
\ge C \big( \|u_0\|_{H^s}\|w_0\|_{H^s}^2 + \| u_0 \|_{H^s} \| v_0 \|_{H^s}^2 \big).
\end{equation}
\end{prop}
\begin{proof}
We set $v_0=0$.
Then, it follows from \eqref{iter} that
\begin{gather}
\begin{split}
&v^{(1)} = u^{(2)} = w^{(2)}=v^{(3)}=0, \\
&v^{(2)} = i \It_{\beta} (\dx w^{(1)} \cdot u^{(1)}),  \\
&u^{(3)} = i \It_{\alpha} (\dx w^{(1)} \cdot v^{(2)}), \quad
w^{(3)} = -i \It_{\gamma} (\dx (u^{(1)} \overline{v^{(2)}})).
\end{split}
\label{iter2}
\end{gather}
For $\xi$, $\eta \in \R$, we set
\[
\Phi (\xi,\eta):=\alpha \xi^2-\beta(\xi-\eta)^2-\gamma \eta^2. 
\]
By \eqref{iter0} and \eqref{iter2},
we have
\[
\begin{split}
\widehat{v^{(2)}}(t',\eta )
&=
i\int_0^te^{-i(t'-t'')\beta\eta^2}
\int_{\R}i\eta_1\widehat{\overline{w^{(1)}}}(t'',\eta_1)\widehat{u^{(1)}}(t'',\eta-\eta_1)d\eta_1dt''\\
&=
-e^{-it'\beta \eta^2}
\int_{\R}\left(\int_0^{t'}e^{-it''\Phi (\eta -\eta_1,-\eta_1)}dt''\right)
\eta_1\widehat{u_0}(\eta-\eta_1)\overline{\widehat{w_0}(-\eta_1)}d\eta_1\\
&=
e^{-it'\beta \eta^2}
\int_{\R}\left(\int_0^{t'}e^{-it''\Phi (\eta +\xi_2,\xi_2)}dt''\right)
\xi_2\widehat{u_0}(\eta+\xi_2)\overline{\widehat{w_0}(\xi_2)}d\xi_2.
\end{split}
\]
Moreover,
we obtain
\begin{equation}
\begin{split}
\widehat{u^{(3)}}(t,\xi)
&=
i\int_0^t e^{-i(t-t')\alpha \xi^2}
\int_{\R}i\xi_1\widehat{w^{(1)}}(t',\xi_1)\widehat{v^{(2)}}(t',\xi-\xi_1)d\xi_1dt'\\
&=
-e^{-it\alpha \xi^2}
\int_{\R}\int_{\R}\left(\int_0^te^{it'\Phi (\xi,\xi_1)}
\int_0^{t'}e^{-it''\Phi(\xi-\xi_1+\xi_2,\xi_2)}dt''dt'\right)\\
&\hspace{20ex}\times
\xi_1\xi_2\widehat{w_0}(\xi_1)\overline{\widehat{w_0}(\xi_2)}
\widehat{u_0}(\xi-\xi_1+\xi_2)d\xi_1d\xi_2.
\end{split}
\label{iter30}
\end{equation}

Now, we give the initial data $u_0$ and $w_0$ as
\begin{equation}
\widehat{u_0}(\xi)= N^{-s}\ee_{[kN,kN+1]}(\xi),\ \ 
\widehat{w_0}(\xi)= N^{-s}\ee_{[N,N+1]}(\xi),
\label{init0}
\end{equation}
where $k$ is a constant will be chosen later.
Then, we have
\begin{equation}
\| u_0 \|_{H^s} \sim \| w_0 \|_{H^s} \sim 1.
\label{iniHs}
\end{equation}

By choosing $k$ appropriately,
we have
\begin{equation}\label{modul}
|\Phi (\xi-\xi_1+\xi_2, \xi_2)|\sim N^2, \quad
|\Phi (\xi, \xi_1)|\sim N^2, \quad
|\Psi (\xi,\xi_1,\xi_2)| \lesssim 1
\end{equation}
for $\xi_1, \xi_2 \in [N,N+1]$ and $\xi-\xi_1+\xi_2 \in [kN,kN+1]$,
where
\begin{equation}
\Psi (\xi,\xi_1,\xi_2) := \Phi (\xi, \xi_1)-\Phi (\xi-\xi_1+\xi_2, \xi_2).
\label{Psi2}
\end{equation}
Here, we prove \eqref{flow_map_estimate} by assuming \eqref{modul}.
It follows from \eqref{iter30} that
\[
\begin{split}
|\widehat{u_3}(t,\xi )|
&=\bigg|\int_{\R}\int_{\R}
\int_0^t\frac{e^{it'\Psi (\xi,\xi_1,\xi_2)}-e^{it'\Phi(\xi,\xi_1)}}{\Phi(\xi-\xi_1+\xi_2,\xi_2)}dt'
\xi_1\xi_2 \\
&\hspace*{100pt}
\times
\widehat{w_0}(\xi_1)\overline{\widehat{w_0}(\xi_2)}
\widehat{u_0}(\xi-\xi_1+\xi_2)d\xi_1d\xi_2\bigg|\\
&\gtrsim 
\bigg|\int_{\R}\int_{\R}
\left(\int_0^te^{it'\Psi (\xi,\xi_1,\xi_2)}dt'\right)
\frac{\xi_1\xi_2}{\Phi(\xi-\xi_1+\xi_2,\xi_2)} \\
&\hspace*{100pt}
\times
\widehat{w_0}(\xi_1)\overline{\widehat{w_0}(\xi_2)}
\widehat{u_0}(\xi-\xi_1+\xi_2)d\xi_1d\xi_2 \bigg|\\
&\ \ \ \ \ \ \ 
-\bigg|\int_{\R}\int_{\R}\frac{e^{it\Phi(\xi,\xi_1)}-1}{\Phi(\xi,\xi_1)\Phi(\xi-\xi_1+\xi_2,\xi_2)}\xi_1\xi_2
\\
&\hspace*{100pt}
\times
\widehat{w_0}(\xi_1)\overline{\widehat{w_0}(\xi_2)}
\widehat{u_0}(\xi-\xi_1+\xi_2)d\xi_1d\xi_2\bigg|\\
&=: {\rm I}- {\rm \II}.
\end{split}
\]
By \eqref{init0},
we have
\[
\left|\int_{\R}\int_{\R}\xi_1\xi_2\widehat{w_0}(\xi_1)\overline{\widehat{w_0}(\xi_2)}
\widehat{u_0}(\xi-\xi_1+\xi_2)d\xi_1d\xi_2\right|
\gtrsim N^{2-3s} \ee_{[kN,kN+\frac{1}{2}]}(\xi).
\]
Hence,
\eqref{modul} yields that
\[
\| {\rm I} \|_{H^s}
\gtrsim tN^{-2} N^{2-3s} \| \langle \xi \rangle^s \ee_{[kN,kN+\frac{1}{2}]}(\xi)\|_{L^2}
\sim tN^{-2s}
\]
for $0<t \ll 1$.
On the other hand, by \eqref{modul}, Young's inequality, and \eqref{init0},
we have
\[
\| {\rm \II} \|_{H^s}
\le N^sN^{-2} \|\widehat{w_0}*\widehat{w_0}*\widehat{u_0}\|_{L^2}
\lesssim N^sN^{-2}\|\widehat{w_0}\|_{L^1}^2\|u_0\|_{L^2}
\sim N^{-2s-2}.
\]
Therefore, we obtain
\begin{equation}
\|u_3(t)\|_{H^s}
\gtrsim tN^{-2s}-N^{-2s-2}
\sim tN^{-2s}
\label{u3Hs}
\end{equation}
provided that $t \gg N^{-2}$.
Then, \eqref{flow_map_estimate} follows from \eqref{iniHs}, \eqref{u3Hs},
and $s<0$.

It remains to prove \eqref{modul} by choosing $k$ appropriately.
If $\xi_1=N+\eps_1$, $\xi_2=N+\eps_2$, and 
$\xi =kN+\eps$ for $0\le \eps, \eps_1, \eps_2\lesssim 1$,
then \eqref{Psi2} yields that
\begin{equation}
\begin{split}
\Psi (\xi,\xi_1,\xi_2)&=
\alpha\xi^2-\gamma \xi_1^2-\alpha (\xi-\xi_1+\xi_2)^2+\gamma \xi_2^2\\
&=(\xi_1-\xi_2)\left\{2\alpha \xi -(\alpha+\gamma)\xi_1+(\alpha -\gamma)\xi_2\right\}\\
&=(\eps_1-\eps_2)\left\{2(\alpha k-\gamma)N
+2\alpha \eps-(\alpha+\gamma)\eps_1+(\alpha -\gamma)\eps_2\right\}
\end{split}
\label{Psi4}
\end{equation}
Therefore, if we choose $k$ as
\[
k=\frac{\gamma}{\alpha},
\]
then the third condition (\ref{modul}) holds.
Furthermore, we have
\begin{equation}
\begin{split}
|\Phi (\xi-\xi_1+\xi_2,\xi_2)|
&=|\alpha (\xi-\xi_1+\xi_2)^2-\beta(\xi-\xi_1)^2-\gamma \xi_2^2|\\
&\sim |\alpha k^2-\beta (k-1)^2-\gamma|N^2\\
&=\frac{|(\alpha -\gamma )\mu|}{\alpha^2} N^2,
\end{split}
\label{Psi5}
\end{equation}
where $\mu$ is as in \eqref{coeff}.
Since the assumption $\mu>0$ implies $\alpha \neq \gamma$,
\eqref{Psi5} shows the first condition in \eqref{modul}.

Finally,
it follows from \eqref{Psi2}, \eqref{Psi4}, and \eqref{Psi5} that
\[
|\Phi (\xi, \xi_1)|=|\Psi (\xi,\xi_1,\xi_2)+\Phi (\xi-\xi_1+\xi_2,\xi_2)|\sim N^2.
\]
We therefore obtain \eqref{modul}.
This concludes the proof of Proposition \ref{notC3}.
\end{proof}
\section*{acknowledgements}
This work was supported by 
JSPS KAKENHI Grant Numbers JP17K14220
and JP20K14342,  
Program to Disseminate Tenure Tracking System from the Ministry of Education, Culture, Sports, Science and Technology, 
and the DFG through the CRC 1283 "Taming uncertainty and
profiting from randomness and low regularity in analysis, stochastics
and their
applications".  

\end{document}